\documentclass[10pt,a4paper]{article}

\usepackage{latexsym,amsfonts,amsmath,amssymb,amsthm,mathrsfs}
\usepackage[dvipdfmx]{color,graphicx}

\newtheorem{theorem}{Theorem}
\newtheorem{lemma}{Lemma}
\newtheorem{algorithm}{Algorithm}
\newtheorem{remark}{Remark}

\newtheorem{definition}{Definition}
\newtheorem{corollary}{Corollary}

\newcommand{\rd}{\, \mathrm{d}}
\newcommand{\bszero}{\boldsymbol{0}}

\newcommand{\bsk}{\boldsymbol{k}}
\newcommand{\bsq}{\boldsymbol{q}}
\newcommand{\bsx}{\boldsymbol{x}}
\newcommand{\bsy}{\boldsymbol{y}}

\newcommand{\bsalpha}{\boldsymbol{\alpha}}
\newcommand{\bsbeta}{\boldsymbol{\beta}}
\newcommand{\bsgamma}{\boldsymbol{\gamma}}
\newcommand{\bstau}{\boldsymbol{\tau}}

\newcommand{\FF}{\mathbb{F}}
\newcommand{\NN}{\mathbb{N}}
\newcommand{\RR}{\mathbb{R}}
\newcommand{\ZZ}{\mathbb{Z}}

\newcommand{\grid}{\mathrm{grid}}
\newcommand{\tr}{\mathrm{tr}}

\newcommand{\wal}{\mathrm{wal}}
\newcommand{\wor}{\mathrm{wor}}

\allowdisplaybreaks

\begin{document}

\title{Richardson extrapolation of polynomial lattice rules}
\author{
Josef Dick\thanks{School of Mathematics and Statistics, The University of New South Wales, Sydney, NSW 2052, Australia (\tt{josef.dick@unsw.edu.au})},
Takashi Goda\thanks{Graduate School of Engineering, The University of Tokyo, 7-3-1 Hongo, Bunkyo-ku, Tokyo 113-8656, Japan (\tt{goda@frcer.t.u-tokyo.ac.jp})},
Takehito Yoshiki\thanks{Department of Applied Mathematics and Physics, Graduate School of Informatics, Kyoto University, Kyoto 606-8561, Japan (\tt{yoshiki.takehito.47x@st.kyoto-u.ac.jp})}}
\date{\today}

\maketitle

\begin{abstract}
We study multivariate numerical integration of smooth functions in weighted Sobolev spaces with dominating mixed smoothness $\alpha\geq 2$ defined over the $s$-dimensional unit cube.
We propose a new quasi-Monte Carlo (QMC)-based quadrature rule, named \emph{extrapolated polynomial lattice rule}, which achieves the almost optimal rate of convergence. Extrapolated polynomial lattice rules are constructed in two steps: i) construction of classical polynomial lattice rules over $\FF_b$ with $\alpha$ consecutive sizes of nodes, $b^{m-\alpha+1},\ldots,b^{m}$, and ii) recursive application of Richardson extrapolation to a chain of $\alpha$ approximate values of the integral obtained by consecutive polynomial lattice rules.

We prove the existence of good extrapolated polynomial lattice rules achieving the almost optimal order of convergence of the worst-case error in Sobolev spaces with general weights.
Then, by restricting to product weights, we show that such good extrapolated polynomial lattice rules can be constructed by the fast component-by-component algorithm under a computable quality criterion. 
The required total construction cost is of order $(s+\alpha)N\log N$, which improves the currently known result for interlaced polynomial lattice rule, that is of order $s\alpha N\log N$. We also study the dependence of the worst-case error bound on the dimension. 

A big advantage of our method compared to interlaced polynomial lattice rules is that the fast QMC matrix vector method can be used in this setting, while still achieving  the same rate of convergence. Such a method was previously not known. 

Numerical experiments for test integrands support our theoretical result.\end{abstract}
\noindent
Keywords: quasi-Monte Carlo, polynomial lattice rule, Richardson extrapolation, component-by-component construction\\
Mathematics Subject Classifications: Primary 65C05; Secondary 65D30, 65D32

\section{Introduction}
In this paper we study numerical integration of smooth functions defined over the $s$-dimensional unit cube.
For an integrable function $f\colon [0,1)^s\to \RR$, we denote the integral of $f$ by
\[ I(f) := \int_{[0,1)^s}f(\bsx)\rd \bsx. \]
We approximate $I(f)$ by a linear algorithm of the form
\[ I(f;P_N,W_N) = \sum_{n=0}^{N-1}w_nf(\bsx_n), \]
where $P_N=\{\bsx_n\colon 0\leq n<N\}\subset [0,1)^s$ is the set of quadrature nodes and $W_N=\{w_n\colon 0\leq n<N\}\subset \RR$ is the set of associated weights.
A quasi-Monte Carlo (QMC) rule is an equal-weight quadrature rule where the weights sum up to 1, i.e., a linear algorithm with the special choice $w_n=1/N$ for all $n$.
Thus, $I(f)$ is simply approximated by
\[ I(f;P_N) = \frac{1}{N}\sum_{n=0}^{N-1}f(\bsx_n). \]
We refer to \cite{DKS13,DPbook,Nbook,SJbook} for comprehensive information on QMC integration.

The quality of a given quadrature rule is often measured by the worst-case error, that is, the worst absolute integration error in the unit ball of a normed function space $V$:
\[ e^{\wor}(V; P_N,W_N) := \sup_{\substack{f\in V\\ \|f\|_V\leq 1}}|I(f;P_N,W_N)-I(f)|, \]
for a general linear algorithm, and 
\[ e^{\wor}(V; P_N) := \sup_{\substack{f\in V\\ \|f\|_V\leq 1}}|I(f;P_N)-I(f)|, \]
for a QMC algorithm.
In this paper, we consider weighted unanchored Sobolev spaces with dominating mixed smoothness $\alpha\geq 2$ as introduced in \cite{DKLNS14}, see Section~\ref{subsec;sobolev} for the details.
For such function spaces consisting of smooth functions, it is possible to construct good QMC integration rules achieving the almost optimal order of convergence $O(N^{-\alpha+\epsilon})$ with arbitrarily small $\epsilon>0$, see for instance \cite{BD09,BDGP11,BDLNP12,D08,G15,GD15,GSY16}.
In particular, so-called interlaced polynomial lattice rules have been recently applied in the context of partial differential equations with random coefficients, see for instance \cite{DKLNS14,DLS16}, due to their low construction cost and weak dependence of the worst-case error on the dimension.

In this paper, we propose an alternative QMC-based quadrature rule, named \emph{extrapolated polynomial lattice rule}, which achieves the almost optimal order of convergence with weak dependence on the dimension and can be constructed at a low computational cost.
Roughly speaking, extrapolated polynomial lattice rules are given by constructing classical polynomial lattice rules with consecutive sizes of nodes and then applying Richardson extrapolation in a recursive way.
Therefore, the resulting quadrature rule is a linear algorithm but not equally weighted.
Our motivation behind introduction of extrapolated polynomial lattice rules lies in so-called fast QMC matrix-vector multiplication which is briefly explained below.

Recently in \cite{DKLS15}, Dick et al.\ consider the problem of approximating  integrals of the form 
\[ \int_{[0,1)^s}f(\bsx A) \rd \bsx,\]
where $\bsx$ is an $1\times s$ row vector, and $A$ is an $s\times t$ real matrix.
They design QMC quadrature nodes $\bsx_0,\ldots,\bsx_{N-1}\in [0,1)^s$ suitably such that the matrix-vector product $XA$, where $X=(\bsx_0^{\top},\ldots,\bsx_{N-1}^{\top})^{\top}$, can be computed in $O(N\log N)$ arithmetic operations by using the fast Fourier transform without requiring any structure in the matrix $A$. This is done by choosing the quadrature nodes such that $X = CP$, where $C$ is a circulant matrix and the matrix $P$ reorders and extends the vector $\mathbf{a}$ when multiplied with $P$. 
The resulting vector $XA=Y=(\bsy_0^{\top},\ldots,\bsy_{N-1}^{\top})^{\top}$ is used to approximate $I(f)$ by 
\[ \frac{1}{N}\sum_{n=0}^{N-1}f(\bsy_n). \]
Their proposed method can be applied to classical polynomial lattice rules, but not to interlaced polynomial lattice rules, since the interlacing destroys the circulant structure.
In fact, it has been an open question whether it is possible to design QMC quadrature nodes which achieve higher order of convergence of the integration error for sufficiently smooth functions, and at the same time, can be used in fast QMC matrix-vector multiplication.
Since extrapolated polynomial lattice rules are just given by a linear combination of classical polynomial lattice rules, we can apply fast QMC matrix-vector multiplication to extrapolated polynomial lattice rules in a straightforward manner, which gives an affirmative solution to the above question.

The remainder of this paper is organized as follows.
In the next section we describe the necessary background and notation, namely, weighted unanchored Sobolev spaces with dominating mixed smoothness, Walsh functions, polynomial lattice rules, and Richardson extrapolation.
In Section~\ref{sec:explr}, we first give the key ingredient for introducing extrapolated polynomial lattice rules, and then study their worst-case error in Sobolev spaces with general weights as well as their dependence on the worst-case error bound on the dimension.
Here we prove the existence of good extrapolated polynomial lattice rules achieving the almost optimal order of convergence.
In Section~\ref{sec:cbc}, we restrict ourselves to the case of product weights and show that the so-called fast component-by-component construction algorithm works for finding good extrapolated polynomial lattice rules.
We conclude this paper with numerical experiments in Section~\ref{sec:experiment}.

\section{Preliminaries}

Throughout this paper, let $\NN$ denote the set of positive integers and $\NN_0=\NN\cup \{0\}$.
Let $b$ be a prime, and $\FF_b$ be the finite field with $b$ elements which is identified with the set $\{0,1,\ldots,b-1\}\subset \ZZ$ equipped with addition and multiplication modulo $b$.
Further, we denote by $\FF_b[x]$ the set of all polynomials over $\FF_b$ and by $\FF_b((x^{-1}))$ the field of formal Laurent series over $\FF_b$.
For $m\in \NN$, we write
\[ G_{b,m} = \{q\in \FF_b[x]\colon \deg(m)<m \}\quad \text{and}\quad G^*_{b,m}=G_{b,m}\setminus \{0\}. \]
It is obvious that $|G_{b,m}|=b^m$ and $|G^*_{b,m}|=b^m-1$.
With a slight abuse of notation, we often identify $n\in \NN_0$, whose finite $b$-adic expansion is given by $n=\nu_0+\nu_1b+\cdots$, with the polynomial over $\FF_b$ given by $n(x)=\nu_0+\nu_1x+\cdots$.

\subsection{Sobolev spaces with dominating mixed smoothness}\label{subsec;sobolev}
We give the definition of weighted Sobolev spaces with dominating mixed smoothness.
Let $\alpha,s\in \NN$, $\alpha\geq 2$, $1\leq q,r\leq \infty$, and let $\bsgamma=(\gamma_u)_{u\subset \NN}$ be a set of non-negative real numbers called weights, which plays a role in moderating the importance of different variables or groups of variables in the function space \cite{SW98}.
Assume that $f\colon [0,1)^s\to \RR$ has partial mixed derivatives up to order $\alpha$ in each variable.
We define the norm on the weighted unanchored Sobolev space with dominating mixed smoothness $\alpha$ by
\begin{align*}
\|f\|_{s,\alpha,\bsgamma,q,r} & := \Bigg( \sum_{u\subseteq \{1,\ldots,s\}}\Bigg( \gamma_u^{-q}\sum_{v\subseteq u}\sum_{\bstau_{u\setminus v}\in \{1,\ldots,\alpha\}^{|u\setminus v|}} \\
& \qquad \qquad \int_{[0,1)^{|v|}}\left| \int_{[0,1)^{s-|v|}}f^{(\bstau_{u\setminus v},\bsalpha_v,\bszero)}(\bsx)\rd \bsx_{-v}\right|^q \rd \bsx_v\Bigg)^{r/q}\Bigg)^{1/r},
\end{align*}
with the obvious modifications if $q$ or $r$ is infinite.
Here $(\bstau_{u\setminus v},\bsalpha_v,\bszero)$ denotes a sequence $\bsbeta=(\beta_j)_{1\leq j\leq s}$ with $\beta_j=\tau_j$ if $j\in u\setminus v$, $\beta_j=\alpha$ if $j\in v$, and $\beta_j=0$ if $j\notin u$.
Further, $f^{(\bstau_{u\setminus v},\bsalpha_v,\bszero)}$ denotes the partial mixed derivative of order $(\bstau_{u\setminus v},\bsalpha_v,\bszero)$ of $f$, and we write $\bsx_v=(x_j)_{j\in v}$ and $\bsx_{-v}=(x_j)_{j\in \{1,\ldots,s\}\setminus v}$.
We denote the Banach-Sobolev space of all such functions with finite norm $\|\cdot\|_{s,\alpha,\bsgamma,q,r}$ by $W_{s,\alpha,\bsgamma,q,r}$.

In what follows, let $B_{\tau}(\cdot)$ denote the Bernoulli polynomial of degree $\tau$.
We put $b_{\tau}(\cdot)=B_{\tau}(\cdot)/\tau!$ and $b_{\tau}=b_{\tau}(0)$.
Further, let $\tilde{b}_{\tau}(\cdot)\colon \RR\to \RR$ denote the one-periodic extension of the polynomial $b_{\tau}(\cdot)\colon [0,1)\to \RR$.
Then, as shown in the proof of \cite[Theorem~3.5]{DKLNS14} we have the following.

\begin{lemma}\label{lem:func_represent}
For any $f\in W_{s,\alpha,\bsgamma,q,r}$, we have a pointwise representation
\[ f(\bsy) = \sum_{u\subseteq \{1,\ldots,s\}}f_u(\bsy_u), \]
where each function $f_u$ depends only on a subset of variables $\bsy_u=(y_j)_{j\in u}$ and is explicitly given by
\begin{align*}
f_u(\bsy_u) & = \sum_{v\subseteq u}\sum_{\bstau_{u\setminus v}\in \{1,\ldots,\alpha\}^{|u\setminus v|}}\prod_{j\in u\setminus v}b_{\tau_j}(y_j) \\
& \qquad \times (-1)^{(\alpha+1)|v|}\int_{[0,1)^s} f^{(\bstau_{u\setminus v},\bsalpha_v,\bszero)}(\bsx) \prod_{j\in v}\tilde{b}_{\alpha}(x_j-y_j) \rd \bsx.
\end{align*}
Furthermore we have
\[ \|f\|_{s,\alpha,\bsgamma,q,r} = \left( \sum_{u\subseteq \{1,\ldots,s\}}\|f_u\|_{s,\alpha,\bsgamma,q,r}^r\right)^{1/r}. \]
\end{lemma}

\subsection{Walsh functions}\label{subsec:walsh}
Here we introduce the definition of Walsh functions and state the result on the decay of Walsh coefficients for functions in $W_{s,\alpha,\bsgamma,q,r}$.

\begin{definition}
For a prime $b$, put $\omega_b=\exp(2\pi i/b)$.
For $k\in \NN_0$ with finite $b$-adic expansion $k=\kappa_0+\kappa_1b+\cdots$, the $k$-th Walsh function ${}_{b}\wal_k\colon [0,1)\to \{1,\omega_b,\ldots,\omega_b^{b-1}\}$ is defined by
\[ {}_{b}\wal_k(x) := \omega_{b}^{\kappa_0 \xi_1+\kappa_1 \xi_2+\cdots},\]
for $x\in [0,1)$ with $b$-adic expansion $x=\xi_1/b+\xi_2/b^2+\cdots$, where this expansion is understood to be unique in the sense that infinitely many of the $\xi_i$ are different from $b-1$.

For $s\geq 2$ and $\bsk=(k_1,\ldots,k_s)\in \NN_0^s$, the $\bsk$-th Walsh function ${}_{b}\wal_{\bsk}\colon [0,1)^s\to \{1,\omega_b,\ldots,\omega_b^{b-1}\}$ is defined by
\[ {}_{b}\wal_{\bsk}(\bsx) := \prod_{j=1}^{s}{}_{b}\wal_{k_j}(x_j) ,\]
for $\bsx=(x_1,\ldots,x_s)\in [0,1)^s$.
\end{definition}
\noindent
Since we shall use Walsh functions in a fixed prime base $b$ in this paper, we omit the subscript and simply write $\wal_k$ or $\wal_{\bsk}$. Note that the system $\{\wal_{\bsk}\colon \bsk\in \NN_0^s\}$ is a complete orthonormal system in $L^2([0,1)^s)$, see \cite[Theorem~A.11]{DPbook}. Thus for $f\in L^2([0,1)^s)$, we have the Walsh expansion of $f$:
\[ \sum_{\bsk\in \NN_0^s}\hat{f}(\bsk)\wal_{\bsk}(\bsx), \]
where $\hat{f}(\bsk)$ denotes the $\bsk$-th Walsh coefficient of $f$ defined by
\[ \hat{f}(\bsk):=\int_{[0,1)^s}f(\bsx)\overline{\wal_{\bsk}(\bsx)}\rd \bsx. \]
Here we note that the integral of $f$ is given by $I(f)=\hat{f}(\bszero)$.

The Walsh coefficients of a function $f\in W_{s,\alpha,\bsgamma,q,r}$ are bounded as follows, see \cite[Theorem~14]{D09} and \cite[Theorem~3.5]{DKLNS14} for the proof.

\begin{lemma}\label{lem:walsh_bound}
For $k\in \NN$, we denote the $b$-adic expansion $k$ by $k=\kappa_1b^{a_1-1}+\cdots+\kappa_vb^{a_v-1}$ with $a_1>\cdots>a_v>0$ and $\kappa_1,\ldots,\kappa_v\in \{1,\ldots,b-1\}$.
We define the metric $\mu_{\alpha}:\NN_0\to \NN_0$ by 
\[ \mu_{\alpha}(k):=a_1+\cdots+a_{\min(v,\alpha)}, \]
and $\mu_{\alpha}(0):=0$. In case of a vector $\bsk=(k_1,\ldots,k_s)\in \NN_0^s$, we define
\[ \mu_{\alpha}(\bsk) := \sum_{j=1}^{s}\mu_{\alpha}(k_j). \]

For a subset $u\subseteq \{1,\ldots,s\}$ and $\bsk_u\in \NN^{|u|}$, the $(\bsk_u,\bszero)$-th Walsh coefficient of a function $f\in W_{s,\alpha,\bsgamma,q,r}$ is bounded by
\[ |\hat{f}(\bsk_u,\bszero)| \leq  \gamma_uC_{\alpha}^{|u|}b^{-\mu_{\alpha}(\bsk_u)}\|f_u\|_{s,\alpha,\bsgamma,q,r}, \]
where 
\begin{align*}
C_{\alpha} & = \max\left( \frac{2}{(2\sin\frac{\pi}{b})^{\alpha}}, \max_{1\leq z\leq \alpha-1}\frac{1}{(2\sin\frac{\pi}{b})^{z}}\right) \\
& \qquad \times \left( 1+\frac{1}{b}+\frac{1}{b(b+1)}\right)^{\alpha-2}\left( 3+\frac{2}{b}+\frac{2b+1}{b-1}\right).
\end{align*}
\end{lemma}

\begin{remark}\label{rem:walsh_bound}
For the special but important case $b=2$, Yoshiki \cite{Y15} proved that the constant $C_{\alpha}$ can be improved to $C_{\alpha}=2^{-1/p'}$ where $p'$ denotes the H\"older conjugate of $q$, i.e., $1\leq q'\leq \infty$ which satisfies $1/q+1/q'=1$.
\end{remark}

\subsection{Polynomial lattice rules}\label{subsec:poly_lattice}
Polynomial lattice point sets are a special construction of QMC quadrature nodes introduced by Niederreiter in \cite{N88}, which are defined as follows.

\begin{definition}
Let $p\in \FF_{b}[x]$ with $\deg(p)=m$ and $\bsq=(q_1,\ldots,q_s)\in (G^*_{b,m})^s$.
We define the map $v_m\colon \FF_b((x^{-1}))\to [0,1)$ by
\[ v_m\left( \sum_{i=w}^{\infty}a_i x^{-i}\right) := \sum_{i=\max\{1,w\}}^{m}a_ib^{-i}. \]
For $0\leq n< b^m$, which is identified with a polynomial over $\FF_b$, put
\[ \bsx_n = \left( v_m\left(\frac{nq_1}{p}\right),\ldots, v_m\left(\frac{nq_s}{p}\right)\right)\in [0,1)^s. \]
Then the point set $P(p,\bsq)=\{\bsx_0,\ldots,\bsx_{b^m-1}\}$ is called a polynomial lattice point set (with modulus $p$ and generating vector $\bsq$).
A QMC rule using the point set $P(p,\bsq)$ as quadrature nodes is called a polynomial lattice rule.
\end{definition}

The concept of dual polynomial lattice plays a key role in the error analysis of polynomial lattice rules.
\begin{definition}
For $k\in \NN_0$ with finite $b$-adic expansion $k=\kappa_0+\kappa_1b+\cdots$, we define the map $\tr_m\colon \NN_0\to G_{b,m}$ by
\[ \tr_m(k) = \kappa_0+\kappa_1x+\cdots +\kappa_{m-1}x^{m-1}. \]
For $p\in \FF_{b}[x]$ with $\deg(p)=m$ and $\bsq=(q_1,\ldots,q_s)\in (G^*_{b,m})^s$, the dual polynomial lattice of $P(p,\bsq)$ is defined by
\[ P^{\perp}(p,\bsq) := \left\{\bsk\in \NN_0^s\colon \tr_m(\bsk)\cdot \bsq \equiv 0  \pmod p \right\}. \]
\end{definition}

\begin{remark}\label{rem:poly_lattice}
For $\bsk\in \NN_0^s$ such that $b^m\mid k_j$ for all $j$, we have $\tr_m(\bsk)=\bszero$.
Thus, regardless of the choice on $p$ and $\bsq$, such $\bsk$ is always included in the dual polynomial lattice $P^{\perp}(p,\bsq)$.
\end{remark}

The following lemma shows the character property of polynomial lattice point sets, see for instance \cite[Lemmas~4.75 and 10.6]{DPbook} for the proof.
\begin{lemma}\label{lem:character}
Let $p\in \FF_{b}[x]$ with $\deg(p)=m$ and $\bsq=(q_1,\ldots,q_s)\in (G^*_{b,m})^s$.
For $\bsk\in \NN_0^s$, we have
\begin{align*}
\sum_{\bsx\in P(p,\bsq)}\wal_{\bsk}(\bsx) = \begin{cases}
b^m & \text{if $\bsk\in P^{\perp}(p,\bsq)$}, \\
0 & \text{otherwise.}
\end{cases}
\end{align*}
\end{lemma}
\noindent
By considering the Walsh expansion of a continuous function $f\colon [0,1)^s \to \RR$ with $\sum_{\bsk\in \NN_0^s}|\hat{f}(\bsk)|<\infty$ and using Lemma~\ref{lem:character}, we obtain
\begin{align}
 I(f; P(p,\bsq)) & = \frac{1}{b^m}\sum_{\bsx\in P(p,\bsq)}\sum_{\bsk\in \NN_0^s}\hat{f}(\bsk)\wal_{\bsk}(\bsx) \nonumber \\
& = \frac{1}{b^m}\sum_{\bsk\in \NN_0^s}\hat{f}(\bsk)\sum_{\bsx\in P(p,\bsq)}\wal_{\bsk}(\bsx) \nonumber \\
& = \sum_{\bsk\in P^{\perp}(p,\bsq)}\hat{f}(\bsk) = I(f)+\sum_{\bsk\in P^{\perp}(p,\bsq)\setminus \{\bszero\}}\hat{f}(\bsk) .\label{eq:poly_lattice_error}
\end{align}

\subsection{Richardson extrapolation}\label{subsec:extrapolation}
Richardson extrapolation is a classical technique to speed up the convergence of a sequence by exploiting the asymptotic expansion of each term, see for instance \cite[Section~1.4]{DRbook} and \cite[Section~3.2.7]{Gbook}. In our current setting, we may have a sequence of polynomial lattice rules with the consecutive sizes of nodes, $b^1,b^2,\ldots$, which means that each term of the sequence corresponds to the approximate value $I(f;P(p,\bsq))$ for some $m\in \NN$. 

To simplify the situation, instead of an infinite sequence, let us consider a chain of $\alpha$ reals $(I^{(1)}_n)_{m-\alpha+1\leq n \leq m}$ with each given by
\begin{align}\label{eq:extra_seq}
 I^{(1)}_n = c_0 + \frac{c_1}{b^n}+\cdots + \frac{c_{\alpha-1}}{b^{(\alpha -1)n}} + R_{\alpha,b^n}, 
\end{align}
where $b>1$, $c_0,\ldots,c_{\alpha-1}\in \RR$ and $R_{\alpha,n}\in O(b^{-\alpha n})$. As shown later in \eqref{eq:decomp}, $I(f;P(p,\bsq))$ has actually such an expansion. In standard notation for extrapolation methods, the reciprocal $1/b^n$ should be regarded as a so-called admissible value of the step parameter $h$ for each term $I^{(1)}_n$. The aim here is to approximate $c_0$ as precisely as possible from the chain $(I^{(1)}_n)_{m-\alpha+1\leq n \leq m}$ without knowing the coefficients $c_1,\ldots,c_{\alpha-1}$.

To do so, let us consider the following recursive application of Richardson extrapolation of successive orders:
For $1\leq \tau< \alpha$, compute
\begin{align*}
I^{(\tau+1)}_{n} = \frac{b^{\tau} I^{(\tau)}_{n}-I^{(\tau)}_{n-1}}{b^{\tau}-1} \qquad \text{for $m-\alpha+\tau < n\leq m$}.
\end{align*}
Regarding this recursion, the following result holds. Although a similar result has been shown, for instance, in \cite{LP17}, we give the proof for self-containedness.

\begin{lemma}\label{lem:extra}
For a given $1\leq \tau \leq \alpha$, let
\[ a_{\nu}^{(\tau)} := \prod_{j=1}^{\nu-1}\left(\frac{-1}{b^j-1}\right)\prod_{j=1}^{\tau-\nu}\left(\frac{b^j}{b^j-1}\right) \qquad \text{for $1\leq \nu\leq \tau$}, \]
where the empty product is set to $1$.
Then we have
\[ I^{(\tau)}_{n} = \sum_{\nu=1}^{\tau}a_{\nu}^{(\tau)} I_{n+1-\nu}^{(1)} \qquad \text{for $m-\alpha+\tau \leq n\leq m$}. \]
\end{lemma}

\begin{proof}
We prove the lemma by induction on $\tau$.
As $a_1^{(1)}=1$, the case $\tau=1$ is trivial. 
Let $1\leq \tau<\alpha$ and suppose that the equality
\[ I^{(\tau)}_{n} = \sum_{\nu=1}^{\tau}a_{\nu}^{(\tau)} I_{n+1-\nu}^{(1)} \]
holds for all $m-\alpha+\tau \leq n\leq m$.
It follows from the definition of $I^{(\tau+1)}_{n}$ that
\begin{align*}
I^{(\tau+1)}_{n} & = \frac{b^{\tau} I^{(\tau)}_{n}-I^{(\tau)}_{n-1}}{b^{\tau}-1} \\
& = \frac{b^\tau}{b^\tau-1} \sum_{\nu=1}^{\tau}a_{\nu}^{(\tau)} I_{n+1-\nu}^{(1)} - \frac{1}{b^\tau-1} \sum_{\nu=1}^{\tau}a_{\nu}^{(\tau)} I_{n-\nu}^{(1)} \\
& = \frac{b^\tau}{b^\tau-1} a_1^{(\tau)} I_{n}^{(1)} - \frac{1}{b^\tau-1} a_{\tau}^{(\tau)} I_{n-\tau}^{(1)} + \sum_{\nu=2}^{\tau}\left(\frac{b^\tau}{b^\tau-1}a_{\nu}^{(\tau)}- \frac{1}{b^\tau-1}a_{\nu-1}^{(\tau)}\right)I_{n+1-\nu}^{(1)},
\end{align*}
for $m-\alpha+\tau < n\leq m$.
For each term on the right-most side above, we have
\begin{align*}
\frac{b^\tau}{b^\tau-1} a_1^{(\tau)} & = \frac{b^\tau}{b^\tau-1}\prod_{j=1}^{\tau-1}\left(\frac{b^j}{b^j-1}\right) = \prod_{j=1}^{\tau}\left(\frac{b^j}{b^j-1}\right) = a_1^{(\tau+1)}, \\
- \frac{1}{b^\tau-1} a_{\tau}^{(\tau)} & = \frac{-1}{b^\tau-1}\prod_{j=1}^{\tau-1}\left(\frac{-1}{b^j-1}\right) = \prod_{j=1}^{\tau}\left(\frac{-1}{b^j-1}\right) = a_{\tau+1}^{(\tau+1)} ,
\end{align*}
and for $2\leq \nu\leq \tau$
\begin{align*}
& \frac{b^\tau}{b^\tau-1}a_{\nu}^{(\tau)}- \frac{1}{b^\tau-1}a_{\nu-1}^{(\tau)} \\
& = \frac{b^\tau}{b^\tau-1} \prod_{j=1}^{\nu-1}\left(\frac{-1}{b^j-1}\right)\prod_{j=1}^{\tau-\nu}\left(\frac{b^j}{b^j-1}\right) - \frac{1}{b^\tau-1} \prod_{j=1}^{\nu-2}\left(\frac{-1}{b^j-1}\right)\prod_{j=1}^{\tau-\nu+1}\left(\frac{b^j}{b^j-1}\right) \\
& = \prod_{j=1}^{\nu-1}\left(\frac{-1}{b^j-1}\right)\prod_{j=1}^{\tau+1-\nu}\left(\frac{b^j}{b^j-1}\right) = a_{\nu}^{(\tau+1)}.
\end{align*}
Thus we have
\begin{align*}
I^{(\tau+1)}_{n} = a_1^{(\tau+1)} I_{n}^{(1)} + a_{\tau+1}^{(\tau+1)} I_{n-\tau}^{(1)} + \sum_{\nu=2}^{\tau}a_{\nu}^{(\tau+1)} I_{n+1-\nu}^{(1)} = \sum_{\nu=1}^{\tau+1}a_{\nu}^{(\tau+1)} I_{n+1-\nu}^{(1)},
\end{align*}
which proves the lemma.
\end{proof}

In particular, this lemma shows that the final value $I^{(\alpha)}_{m}$ is given by
\begin{align}\label{eq:extra_final}
    I^{(\alpha)}_{m} = \sum_{\tau=1}^{\alpha}a_{\tau}^{(\alpha)} I_{m-\tau+1}^{(1)}.
\end{align}
Regarding the coefficients $a_{\nu}^{(\tau)}$, the following property holds:
\begin{lemma}\label{lem:extra2}
For any $1\leq \tau \leq \alpha$, we have
\[ \sum_{\nu=1}^{\tau}a_{\nu}^{(\tau)} = 1\qquad \text{and}\qquad \sum_{\nu=1}^{\tau}a_{\nu}^{(\tau)}b^{w(\nu-1)} = 0 \qquad \text{for $1\leq w\leq \tau-1$}. \]
\end{lemma}

\begin{proof}
We prove the lemma by induction on $\tau$.
As $a_1^{(1)}=1$, the case $\tau=1$ is trivial. 
Suppose that the claim of this lemma holds for some $1\leq \tau<\alpha$.
Using the recursions appearing in the proof of Lemma~\ref{lem:extra}, we have
\begin{align*}
\sum_{\nu=1}^{\tau+1}a_{\nu}^{(\tau+1)} & = a_{1}^{(\tau+1)}+\sum_{\nu=2}^{\tau}a_{\nu}^{(\tau+1)}+a_{\tau+1}^{(\tau+1)} \\
& = \frac{b^\tau}{b^\tau-1} a_1^{(\tau)}+\sum_{\nu=2}^{\tau}\left( \frac{b^\tau}{b^\tau-1}a_{\nu}^{(\tau)}- \frac{1}{b^\tau-1}a_{\nu-1}^{(\tau)} \right) - \frac{1}{b^\tau-1} a_{\tau}^{(\tau)}\\
& = \sum_{\nu=1}^{\tau}\left( \frac{b^\tau}{b^\tau-1}-\frac{1}{b^\tau-1}\right)a_{\nu}^{(\tau)} = \sum_{\nu=1}^{\tau}a_{\nu}^{(\tau)}=1.
\end{align*}
Similarly, for $1\leq w\leq \tau$, we have
\begin{align*}
\sum_{\nu=1}^{\tau+1}a_{\nu}^{(\tau+1)}b^{w(\nu-1)} & = a_{1}^{(\tau+1)}+\sum_{\nu=2}^{\tau}a_{\nu}^{(\tau+1)}b^{w(\nu-1)}+a_{\tau+1}^{(\tau+1)}b^{w\tau} \\
& = \frac{b^\tau}{b^\tau-1} a_1^{(\tau)}+\sum_{\nu=2}^{\tau}\left( \frac{b^{\tau+w(\nu-1)}}{b^\tau-1}a_{\nu}^{(\tau)}- \frac{b^{w(\nu-1)}}{b^\tau-1}a_{\nu-1}^{(\tau)} \right) - \frac{b^{w\tau}}{b^\tau-1} a_{\tau}^{(\tau)}\\
& = \sum_{\nu=1}^{\tau}\left( \frac{b^{\tau+w(\nu-1)}}{b^\tau-1}-\frac{b^{w\nu}}{b^\tau-1}\right)a_{\nu}^{(\tau)} \\
& = \frac{b^\tau-b^w}{b^\tau-1}\sum_{\nu=1}^{\tau}a_{\nu}^{(\tau)}b^{w(\nu-1)}=0,
\end{align*}
where the last equality follows from the induction assumption for $1\leq w\leq \tau-1$, and is trivial for $w=\tau$.
\end{proof}

Using these results, we further have the following.
\begin{corollary}\label{cor:extra}
Using the notation above, we have
\[ I^{(\alpha)}_{m} = c_0 + \sum_{\tau=1}^{\alpha}a_{\tau}^{(\alpha)}R_{\alpha,b^{m-\tau+1}}. \]
\end{corollary}

\begin{proof}
Plugging the expression \eqref{eq:extra_seq} into \eqref{eq:extra_final} and then using Lemma~\ref{lem:extra2}, we have
\begin{align*}
 I^{(\alpha)}_{m} & = \sum_{\tau=1}^{\alpha}a_{\tau}^{(\alpha)} \left( c_0 + \sum_{w=1}^{\alpha-1}\frac{c_w}{b^{w(m-\tau+1)}} + R_{\alpha,b^{m-\tau+1}}\right) \\
& = c_0 \sum_{\tau=1}^{\alpha}a_{\tau}^{(\alpha)} + \sum_{w=1}^{\alpha-1}\sum_{\tau=1}^{\alpha}a_{\tau}^{(\alpha)}\frac{c_w}{b^{w(m-\tau+1)}}+ \sum_{\tau=1}^{\alpha}a_{\tau}^{(\alpha)}R_{\alpha,b^{m-\tau+1}} \\
& = c_0 \sum_{\tau=1}^{\alpha}a_{\tau}^{(\alpha)} + \sum_{w=1}^{\alpha-1}\frac{c_w}{b^{wm}}\sum_{\tau=1}^{\alpha}a_{\tau}^{(\alpha)}b^{w(\tau-1)}+ \sum_{\tau=1}^{\alpha}a_{\tau}^{(\alpha)}R_{\alpha,b^{m-\tau+1}}\\
& = c_0 + \sum_{\tau=1}^{\alpha}a_{\tau}^{(\alpha)}R_{\alpha,b^{m-\tau+1}}.
\end{align*}
This completes the proof.
\end{proof}

\section{Extrapolated polynomial lattice rules}\label{sec:explr}

The main idea for coming up with extrapolated polynomial lattice rules is to look at the approximate value of a polynomial lattice rule, as shown in \eqref{eq:poly_lattice_error}, in the following way:
\begin{align*}
I(f; P(p,\bsq)) & = I(f)+ \sum_{\substack{\bsk\in P^{\perp}(p,\bsq)\setminus \{\bszero\}\\ \exists j\colon b^m\nmid k_j}} \hat{f}(\bsk)+\sum_{\substack{\bsk\in P^{\perp}(p,\bsq)\setminus \{\bszero\}\\ \forall j\colon b^m\mid k_j}} \hat{f}(\bsk) \\
& = I(f) + \sum_{\substack{\bsk\in P^{\perp}(p,\bsq)\setminus \{\bszero\}\\ \exists j\colon b^m\nmid k_j}} \hat{f}(\bsk)+\sum_{\bsk\in \NN_0^s\setminus \{\bszero\}} \hat{f}(b^m\bsk),
\end{align*}
where the second equality follows from Remark~\ref{rem:poly_lattice}.
By considering the character property of regular grids
\[ P_{\grid,b^m} = \left\{ \left( \frac{n_1}{b^m},\ldots,\frac{n_s}{b^m}\right)\in [0,1)^s\colon 0\leq n_1,\ldots,n_s< b^m\right\},\]
we see that the third term in the last expression is nothing but the approximation error of $f$ when using $P_{\grid,b^m}$ as quadrature nodes in a QMC integration.
Therefore we have
\[ I(f; P(p,\bsq)) = I(f) + \sum_{\substack{\bsk\in P^{\perp}(p,\bsq)\setminus \{\bszero\}\\ \exists j\colon b^m\nmid k_j}} \hat{f}(\bsk)+ \left(I(f;P_{\grid,b^m}) -I(f)\right). \]
Plugging in the Euler-Maclaurin formula for $I(f;P_{\grid,b^m})$, which is shown later in Theorem~\ref{thm:euler-maclaurin}, into the right-hand side above, we obtain
\begin{align}\label{eq:decomp}
 I(f; P(p,\bsq)) = I(f) + \sum_{\substack{\bsk\in P^{\perp}(p,\bsq)\setminus \{\bszero\}\\ \exists j\colon b^m\nmid k_j}} \hat{f}(\bsk)+ \sum_{\tau=1}^{\alpha -1}\frac{c_{\tau}(f)}{b^{\tau m}} + R_{s,\alpha,b^m}, 
\end{align}
where $c_{\tau}(f)$ depends only on $f$ and $\tau$, and the remainder term $R_{s,\alpha,b^m}$ is proven to decay with order $b^{-\alpha m}$.

Now suppose that we have polynomial lattice rules with consecutive sizes of nodes, $b^{m-\alpha+1},b^{m-\alpha+2},\ldots,b^m$.
For ease of notation, we denote by $P_{b^n}$ a polynomial lattice point set with the number of nodes equal to $b^n$, and by $P^{\perp}_{b^n}$ the dual polynomial lattice of $P_{b^n}$.
Then we can obtain a chain of $\alpha$ approximate values of the integral, i.e., $I(f;P_{b^{m-\alpha+1}}),\ldots,I(f;P_{b^m})$.
By applying Richardson extrapolation in a recursive way as described in Section~\ref{subsec:extrapolation}, it follows from Lemma~\ref{lem:extra}, Corollary~\ref{cor:extra} and \eqref{eq:decomp} that the final value is given by
\begin{align}\label{eq:extra_approximation}
 \sum_{\tau=1}^{\alpha}a_{\tau}^{(\alpha)} I(f;P_{b^{m-\tau+1}}) = I(f) + \sum_{\tau=1}^{\alpha}a_{\tau}^{(\alpha)}\left( \sum_{\substack{\bsk\in P^{\perp}_{b^{m-\tau+1}}\setminus \{\bszero\}\\ \exists j\colon b^{m-\tau+1}\nmid k_j}} \hat{f}(\bsk)+ R_{s,\alpha,b^{m-\tau+1}}\right).
\end{align}
If we can construct good polynomial lattice rules such that the inner sum on the right-hand side of \eqref{eq:extra_approximation} decays with order $b^{-(\alpha-\epsilon)m}$ (with arbitrarily small $\epsilon >0$) for any function $f\in W_{s,\alpha,\bsgamma,q,r}$, the integration error
\[\sum_{\tau=1}^{\alpha}a_{\tau}^{(\alpha)} I(f;P_{b^{m-\tau+1}})-I(f)\]
decays with the almost optimal order.
(Note that we use $N=b^{m-\alpha+1}+\cdots + b^m$ quadrature nodes in total, which does not affects the order of convergence.)
This is our key observation for introducing extrapolated polynomial lattice rules.

In what follows, we start with showing the worst-case error bound of extrapolated polynomial lattice rules, and then in Section~\ref{subsec:euler-maclaurin}, we prove the Euler-Maclaurin formula on the regular grid quadrature.
In Section~\ref{subsec:existence}, we prove the existence of such good polynomial lattice rules for $W_{s,\alpha,\bsgamma,q,r}$ with general weights $\bsgamma=(\gamma_u)_{u\subset \NN}$.
In Section~\ref{sec:cbc}, by restricting to product weights, i.e., the case where the weights are given by the form $\gamma_u=\prod_{j\in u}\gamma_j$ for a sequence of reals $(\gamma_j)_{j\in \NN}$, we show that good polynomial lattice rules can be constructed by the fast component-by-component (CBC) algorithm.

\subsection{Worst-case error bound}\label{subsec:worst-case}
Using the equality \eqref{eq:extra_approximation}, the absolute integration error of an extrapolated polynomial lattice rule is bounded by
\begin{align}\label{eq:bound_error}
 \left| \sum_{\tau=1}^{\alpha}a_{\tau}^{(\alpha)} I(f;P_{b^{m-\tau+1}}) -I(f)\right| \leq \sum_{\tau=1}^{\alpha}|a_{\tau}^{(\alpha)}|\left( \sum_{\substack{\bsk\in P^{\perp}_{b^{m-\tau+1}}\setminus \{\bszero\}\\ \exists j\colon b^{m-\tau+1}\nmid k_j}} |\hat{f}(\bsk)|+ |R_{s,\alpha,b^{m-\tau+1}}|\right). 
\end{align}
In the following, we write
\[ P_{b^{m-\tau+1},u}^{\perp}= \left\{\bsk_u\in \NN^{|u|} \colon (\bsk_u,\bszero)\in P_{b^{m-\tau+1}}^{\perp} \right\}, \]
for a subset $\emptyset \neq u\subseteq \{1,\ldots,s\}$. Note that we have 
\[ P_{b^{m-\tau+1}}^{\perp}\setminus \{\bszero\} = \bigcup_{\emptyset \neq u\subseteq \{1,\ldots,s\}}P_{b^{m-\tau+1},u}^{\perp}. \]
We now obtain a worst-case error bound as follows.

\begin{theorem}\label{thm:bound_wrst_error}
Let $\alpha,s\in \NN$, $\alpha\geq 2$, $1\leq q,r\leq \infty$, and let $\bsgamma=(\gamma_u)_{u\subset \NN}$ be a set of weights.
Let $q'$ and $r'$ be the H\"older conjugates of $q$ and $r$, respectively.
For $m\geq \alpha$, we have
\begin{align*}
& \sup_{\substack{f\in W_{s,\alpha,\bsgamma,q,r}\\ \|f\|_{s,\alpha,\bsgamma,q,r}\leq 1}}\left| \sum_{\tau=1}^{\alpha}a_{\tau}^{(\alpha)} I(f;P_{b^{m-\tau+1}}) -I(f)\right| \\
& \qquad \qquad \leq \sum_{\tau=1}^{\alpha}|a_{\tau}^{(\alpha)}| \left( B_{\bsgamma,r}(P_{b^{m-\tau+1}})+\frac{H_{s,\bsgamma,q,r}}{b^{\alpha(m-\tau+1)}}\right),
\end{align*}
where 
\[ B_{\bsgamma,r}(P_{b^{m-\tau+1}}) = \Bigg(\sum_{\emptyset \neq u\subseteq \{1,\ldots,s\}}\Bigg(\gamma_uC_{\alpha}^{|u|}\sum_{\substack{\bsk_u\in P_{b^{m-\tau+1},u}^{\perp}\\ \exists j\in u\colon b^m\nmid k_j}}b^{-\mu_{\alpha}(\bsk_u)}\Bigg)^{r'} \Bigg)^{1/r'}, \]
and 
\[ H_{s,\bsgamma,q,r} = \Bigg(\sum_{u\subseteq \{1,\ldots,s\}}\gamma_u^{r'}(\alpha+1)^{|u|r'/q'}D_{\alpha}^{r'|u|}\Bigg)^{1/r'}, \]
with $D_{\alpha}=\max\left\{|b_1|,\ldots,|b_{\alpha-1}|,\sup_{x\in [0,1)}|\tilde{b}_{\alpha}(x)|\right\}.$
\end{theorem}

\begin{proof}
Let us consider the inner sum on the right-hand side of \eqref{eq:bound_error} first.
Using the bound on the Walsh coefficient in Lemma~\ref{lem:walsh_bound} and H\"older inequality, we have
\begin{align*}
\sum_{\substack{\bsk\in P^{\perp}_{b^{m-\tau+1}}\setminus \{\bszero\}\\ \exists j\colon b^{m-\tau+1} \nmid k_j}} |\hat{f}(\bsk)| & = \sum_{\emptyset \neq u\subseteq \{1,\ldots,s\}}\sum_{\substack{\bsk_u\in P^{\perp}_{b^{m-\tau+1},u}\setminus \{\bszero\}\\ \exists j\in u\colon b^{m-\tau+1}\nmid k_j}} |\hat{f}(\bsk_u,\bszero)| \\
& \leq \sum_{\emptyset \neq u\subseteq \{1,\ldots,s\}}\|f_u\|_{s,\alpha,\bsgamma,q,r}\gamma_uC_{\alpha}^{|u|} \sum_{\substack{\bsk_u\in P^{\perp}_{b^{m-\tau+1},u}\setminus \{\bszero\}\\ \exists j\in u\colon b^{m-\tau+1}\nmid k_j}} b^{-\mu_{\alpha}(\bsk_u)} \\
& \leq \Bigg( \sum_{\emptyset \neq u\subseteq \{1,\ldots,s\}}\|f_u\|_{s,\alpha,\bsgamma,q,r}^r\Bigg)^{1/r} \\
& \qquad \times \Bigg( \sum_{\emptyset \neq u\subseteq \{1,\ldots,s\}}\Bigg(\gamma_uC_{\alpha}^{|u|} \sum_{\substack{\bsk_u\in P^{\perp}_{b^{m-\tau+1},u}\setminus \{\bszero\}\\ \exists j\in u\colon b^{m-\tau+1}\nmid k_j}} b^{-\mu_{\alpha}(\bsk_u)}\Bigg)^{r'}\Bigg)^{1/r'} \\
& \leq \|f\|_{s,\alpha,\bsgamma,q,r}B_{\bsgamma,r}(P_{b^{m-\tau+1}}).
\end{align*}
Regarding the bound on $R_{\alpha,b^{m-\tau+1}}$, it follows from Theorem~\ref{thm:euler-maclaurin} below that
\[ |R_{s,\alpha,b^{m-\tau+1}}| \leq \frac{\|f\|_{s,\alpha,\bsgamma,q,r} H_{s,\bsgamma,q,r}}{b^{\alpha(m-\tau+1)}}.\]
Plugging these bounds into the right-hand side of \eqref{eq:bound_error} and then taking the supremum among $f\in W_{s,\alpha,\bsgamma,q,r}$ such that $\|f\|_{s,\alpha,\bsgamma,q,r}\leq 1$, the result follows. 
\end{proof}

\begin{remark}
As already pointed out in \cite{DKLNS14}, since we have $B_{\bsgamma,r}(P_{b^{m-\tau+1}})\leq B_{\bsgamma,\infty}(P_{b^{m-\tau+1}})$ and $H_{\bsgamma,q,r}\leq H_{\bsgamma,q,\infty}$ for any $r$, it is convenient to work with an upper bound which can be obtained by setting $r=\infty$ and thus $r'=1$. In the rest of this paper, we always consider the case $r=\infty$. The bound $B_{\bsgamma, r}$ is used below to construct good generating vectors for polynomial lattice rules. The choice $r'=1$ simplifies the computation of $B_{\bsgamma, r}$.
\end{remark}
\subsection{Euler-Maclaurin formula for regular grid quadrature}\label{subsec:euler-maclaurin}
Here we show the Euler-Maclaurin formula on $I(f;P_{\grid,N})$, where
\[ P_{\grid,N} = \left\{ \left( \frac{n_1}{N},\ldots,\frac{n_s}{N}\right)\in [0,1)^s\colon 0\leq n_1,\ldots,n_s< N\right\}. \]
As preparation, we prove the following lemma.
\begin{lemma}\label{lem:bernoulli_sum}
For $\tau, N\in \NN$ and $x\in [0,1)$, we have
\[ \frac{1}{N}\sum_{n=0}^{N-1}b_{\tau}\left( \frac{n}{N}\right) = \frac{b_{\tau}}{N^{\tau}} \quad \text{and} \quad \frac{1}{N}\sum_{n=0}^{N-1}\tilde{b}_{\tau}\left( x-\frac{n}{N}\right) = \frac{\tilde{b}_{\tau}(Nx)}{N^{\tau}}.\]
\end{lemma}

\begin{proof}
For $\tau=1$, we obtain the results by direct calculation, which is omitted here.
We assume $\tau\ge 2$.
By using the Fourier series of $b_{\tau}$, we have
\begin{align*}
\frac{1}{N}\sum_{n=0}^{N-1}b_{\tau}\left( \frac{n}{N}\right) & = \frac{1}{N}\sum_{n=0}^{N-1}\frac{-1}{(2\pi i)^{\tau}}\sum_{h\in \ZZ\setminus \{0\}}\frac{e^{2\pi i hn/N}}{h^{\tau}} \\
& = \frac{-1}{(2\pi i)^{\tau}}\sum_{h\in \ZZ\setminus \{0\}}\frac{1}{h^{\tau}}\left(\frac{1}{N}\sum_{n=0}^{N-1}e^{2\pi i hn/N}\right) \\
& = \frac{-1}{(2\pi i)^{\tau}}\sum_{\substack{h\in \ZZ\setminus \{0\}\\ N\mid h}}\frac{1}{h^{\tau}} = \frac{-1}{(2\pi i)^{\tau}}\sum_{h\in \ZZ\setminus \{0\}}\frac{1}{(hN)^{\tau}}= \frac{b_{\tau}}{N^{\tau}},
\end{align*}
which completes the proof of the first equality.
Since the second equality can be proven in exactly the same way by using the Fourier series of $\tilde{b}_{\tau}$, we omit the proof.
\end{proof}

As shown in Lemma~\ref{lem:func_represent}, we have the following pointwise representation for a function $f\in W_{s,\alpha,\bsgamma,q,r}$:
\begin{align}\label{eq:func_represent}
f(\bsy) & =  \sum_{u\subseteq \{1,\ldots,s\}}\sum_{v\subseteq u}\sum_{\bstau_{u\setminus v}\in \{1,\ldots,\alpha\}^{|u\setminus v|}}\prod_{j\in u\setminus v}b_{\tau_j}(y_j) \nonumber \\
& \qquad \times (-1)^{(\alpha+1)|v|}\int_{[0,1)^s} f^{(\bstau_{u\setminus v},\bsalpha_v,\bszero)}(\bsx) \prod_{j\in v}\tilde{b}_{\alpha}(x_j-y_j) \rd \bsx.
\end{align}
By using Lemma~\ref{lem:bernoulli_sum}, we obtain the Euler-Maclaurin formula on $I(f;P_{\grid,N})$.

\begin{theorem}\label{thm:euler-maclaurin}
For $f\in W_{s,\alpha,\bsgamma,q,r}$, we have
\[ I(f;P_{\grid,N}) = I(f)+\sum_{\tau=1}^{\alpha -1}\frac{c_{\tau}(f)}{N^{\tau}} + R_{s,\alpha,N}, \]
where $c_{\tau}(f)$ depends only on $f$ and $\tau$, and is given by
\[ c_{\tau}(f) =  \sum_{\substack{\bstau \in \{0,1,\ldots,\alpha-1\}^{s}\\ |\bstau|_1=\tau}}\prod_{j=1}^{s}b_{\tau_j} \int_{[0,1)^{s}}f^{(\bstau)}(\bsx)\rd \bsx \]
with $ |\bstau|_1 = \sum_{j=1}^{s}|\tau_j|$. Further we have 
\[ |R_{s,\alpha,N}| \leq \frac{\|f\|_{s,\alpha,\bsgamma,q,r} H_{s,\bsgamma,q,r}}{N^{\alpha}},\]
where $H_{s,\bsgamma,q,r}$ is given as in Theorem~\ref{thm:bound_wrst_error}.
\end{theorem}

\begin{proof}
Plugging the representation \eqref{eq:func_represent} into $I(f;P_{\grid,N})$ and using Lemma~\ref{lem:bernoulli_sum}, we have
\begin{align*}
I(f;P_{\grid,N}) & = \frac{1}{N^s}\sum_{n_1=0}^{N-1}\cdots\sum_{n_s=0}^{N-1}f\left( \frac{n_1}{N},\ldots,\frac{n_s}{N}\right) \\
& = \sum_{u\subseteq \{1,\ldots,s\}}\sum_{v\subseteq u}\sum_{\bstau_{u\setminus v}\in \{1,\ldots,\alpha\}^{|u\setminus v|}}\prod_{j\in u\setminus v}\frac{1}{N}\sum_{n_j=0}^{N-1}b_{\tau_j}\left(\frac{n_j}{N}\right) \\
& \qquad \times (-1)^{(\alpha+1)|v|}\int_{[0,1)^s}f^{(\bstau_{u\setminus v},\bsalpha_{v},\bszero)}(\bsx) \prod_{j\in v}\frac{1}{N}\sum_{n_j=0}^{N-1}\tilde{b}_{\alpha}\left(x_j-\frac{n_j}{N}\right)\rd \bsx \\
& = \sum_{u\subseteq \{1,\ldots,s\}}\sum_{v\subseteq u}\sum_{\bstau_{u\setminus v}\in \{1,\ldots,\alpha\}^{|u\setminus v|}}\frac{1}{N^{|\bstau_{u\setminus v}|_1+\alpha|v|}}\prod_{j\in u\setminus v}b_{\tau_j} \\
& \qquad \times (-1)^{(\alpha+1)|v|}\int_{[0,1)^{s}}f^{(\bstau_{u\setminus v},\bsalpha_{v},\bszero)}(\bsx)\prod_{j\in v}\tilde{b}_{\alpha}(Nx_j)\rd \bsx.
\end{align*}
Let us reorder the summands with respect to the value of $|\bstau_{u\setminus v}|_1+\alpha|v|$, which appears in the exponent of $N$. If $|\bstau_{u\setminus v}|_1+\alpha|v|=0$, we must have $u=v=\emptyset$ and the corresponding summand is nothing but $I(f)$. If $|\bstau_{u\setminus v}|_1+\alpha|v|=\tau$ with $1\leq \tau<\alpha$, we must have $v=\emptyset$ and thus 
\begin{align*}
 c_{\tau}(f) & =  \sum_{\emptyset \neq u\subseteq \{1,\ldots,s\}}\sum_{\substack{\bstau_u\in \{1,\ldots,\alpha-1\}^{|u|}\\ |\bstau_u|_1=\tau}}\prod_{j\in u}b_{\tau_j} \int_{[0,1)^s}f^{(\bstau_u,\bszero)}(\bsx)\rd \bsx \\
& = \sum_{\substack{\bstau \in \{0,1,\ldots,\alpha-1\}^{s}\\ |\bstau|_1=\tau}}\prod_{j=1}^{s}b_{\tau_j} \int_{[0,1)^{s}}f^{(\bstau)}(\bsx)\rd \bsx.
\end{align*}
The other summands have the exponents $|\bstau_{u\setminus v}|_1+\alpha|v| \geq \alpha$ and belong to $R_{s,\alpha,N}$. 

Next we prove the bound on $R_{s,\alpha,N}$.
From the above argument, it is obvious that $R_{s,\alpha,N}$ is bounded by
\begin{align*}
 |R_{s,\alpha,N}| & \leq \frac{1}{N^{\alpha}}\Bigg|\sum_{u\subseteq \{1,\ldots,s\}}\sum_{v\subseteq u}\sum_{\bstau_{u\setminus v}\in \{1,\ldots,\alpha\}^{|u\setminus v|}}\prod_{j\in u\setminus v}b_{\tau_j}  \\
& \qquad \times (-1)^{(\alpha+1)|v|}\int_{[0,1)^{s}}f^{(\bstau_{u\setminus v},\bsalpha_{v},\bszero)}(\bsx)\prod_{j\in v}\tilde{b}_{\alpha}(Nx_j)\rd \bsx \Bigg| .
\end{align*}
By applying H\"older's inequality, we have
\begin{align*}
& \Bigg|\int_{[0,1)^{s}}f^{(\bstau_{u\setminus v},\bsalpha_{v},\bszero)}(\bsx)\prod_{j\in v}\tilde{b}_{\alpha}(Nx_j)\rd \bsx \Bigg| \\
& \quad \leq \int_{[0,1)^{|v|}}\Bigg|\int_{[0,1)^{s-|v|}} f^{(\bstau_{u\setminus v},\bsalpha_{v},\bszero)}(\bsx)\rd \bsx_{-v}\Bigg| \cdot \Bigg|\prod_{j\in v}\tilde{b}_{\alpha}(Nx_j)\Bigg|\rd \bsx_v \\
& \quad \leq D_{\alpha}^{|v|}\Bigg(\int_{[0,1)^{|v|}}\Bigg|\int_{[0,1)^{s-|v|}} f^{(\bstau_{u\setminus v},\bsalpha_{v},\bszero)}(\bsx)\rd \bsx_{-v}\Bigg| ^q\rd \bsx_v\Bigg)^{1/q},
\end{align*}
for $1\leq q\leq \infty$. Using the above inequality and H\"older's inequality twice, we obtain
\begin{align*}
& |R_{s,\alpha,N}| \\
& \leq \frac{1}{N^{\alpha}}\sum_{u\subseteq \{1,\ldots,s\}}\sum_{v\subseteq u}\sum_{\bstau_{u\setminus v}\in \{1,\ldots,\alpha\}^{|u\setminus v|}}D_{\alpha}^{|u|}  \\
& \qquad \times \Bigg(\int_{[0,1)^{|v|}}\Bigg|\int_{[0,1)^{s-|v|}} f^{(\bstau_{u\setminus v},\bsalpha_{v},\bszero)}(\bsx)\rd \bsx_{-v}\Bigg| ^q\rd \bsx_v\Bigg)^{1/q} \\
& \leq \frac{1}{N^{\alpha}}\sum_{u\subseteq \{1,\ldots,s\}}\Bigg(\sum_{v\subseteq u}\sum_{\bstau_{u\setminus v}\in \{1,\ldots,\alpha\}^{|u\setminus v|}}\gamma_u^{q'}D_{\alpha}^{q'|u|}\Bigg)^{1/q'}  \\
& \qquad \times \Bigg(\sum_{v\subseteq u}\sum_{\bstau_{u\setminus v}\in \{1,\ldots,\alpha\}^{|u\setminus v|}}\gamma_u^{-q}\int_{[0,1)^{|v|}}\Bigg|\int_{[0,1)^{s-|v|}} f^{(\bstau_{u\setminus v},\bsalpha_{v},\bszero)}(\bsx)\rd \bsx_{-v}\Bigg| ^q\rd \bsx_v\Bigg)^{1/q} \\
& \leq \frac{1}{N^{\alpha}}\Bigg(\sum_{u\subseteq \{1,\ldots,s\}}\gamma_u^{r'}(\alpha+1)^{|u|r'/q'}D_{\alpha}^{r'|u|}\Bigg)^{1/r'}  \\
& \qquad \times \Bigg(\sum_{u\subseteq \{1,\ldots,s\}} \Bigg(\gamma_u^{-q}\sum_{v\subseteq u}\sum_{\bstau_{u\setminus v}\in \{1,\ldots,\alpha\}^{|u\setminus v|}} \\
& \qquad \qquad \qquad \int_{[0,1)^{|v|}}\Bigg|\int_{[0,1)^{s-|v|}} f^{(\bstau_{u\setminus v},\bsalpha_{v},\bszero)}(\bsx)\rd \bsx_{-v}\Bigg| ^q\rd \bsx_v\Bigg)^{r/q}\Bigg)^{1/r} \\
& = \frac{\|f\|_{s,\alpha,\bsgamma,q,r}H_{s,\bsgamma,q,r}}{N^{\alpha}}.
\end{align*}
This completes the proof of this theorem.
\end{proof}

\subsection{Existence results}\label{subsec:existence}
Here we prove the existence of good extrapolated polynomial lattice rules which achieve the almost optimal order of convergence.
Since each point set $P_{b^{m-\tau+1}}$ can be constructed independently, it suffices to prove the existence of a good polynomial lattice rule of size $b^m$ which achieves the almost optimal order of the term $B_{\bsgamma,\infty}(P_{b^m})$ for any $m\in \NN$.
In order to emphasize the role of the modulus $p$ and generating vector $\bsq$, instead of $B_{\bsgamma,\infty}(P_{b^m})$ we write
\[ B_{\bsgamma}(p,\bsq) =  \sum_{\emptyset \neq u\subseteq \{1,\ldots,s\}}\gamma_uC_{\alpha}^{|u|}\sum_{\substack{\bsk_u\in P_u^{\perp}(p,\bsq)\\ \exists j\in u\colon b^m\nmid k_j}}b^{-\mu_{\alpha}(\bsk_u)}, \]
where $m=\deg(p)$.
First we recall the following auxiliary result. See \cite[Lemma~7]{G16} for the proof.
\begin{lemma}\label{lem:sum_mu_alpha}
For $\alpha\geq 2$ and $1/\alpha<\lambda\leq 1$, we have
\[ \sum_{k=1}^{\infty}b^{-\lambda\mu_{\alpha}(k)} = \sum_{w=1}^{\alpha-1}\prod_{i=1}^{w}\left( \frac{b-1}{b^{\lambda i}-1}\right)+ \left( \frac{b^{\lambda \alpha}-1}{b^{\lambda \alpha}-b}\right)\prod_{i=1}^{\alpha}\left( \frac{b-1}{b^{\lambda i}-1}\right) =: E_{\alpha,\lambda}. \]
\end{lemma}

Now we prove the existence result.
\begin{theorem}\label{thm:existence}
Let $p\in \FF_b[x]$ with $\deg(p)=m$ be irreducible.
For a set of weights $\bsgamma=(\gamma_u)_{u\subset \NN}$, there exists at least one $\bsq^*=(q_1^*,\ldots,q_s^*)\in (G^*_{b,m})^s$ such that
\begin{align*}
B_{\bsgamma}(p,\bsq^*) \leq \frac{1}{(b^m-1)^{1/\lambda}}\left[ \sum_{\emptyset \neq u\subseteq \{1,\ldots,s\}}\gamma_u^{\lambda}C_{\alpha}^{\lambda|u|}E_{\alpha,\lambda}^{|u|}\right]^{1/\lambda}
\end{align*}
holds for any $1/\alpha<\lambda\leq 1$.
\end{theorem}

\begin{proof}
Let $\bsq^*$ be given by
\[ \bsq^* = \arg\min_{\bsq \in (G^*_{b,m})^s}B_{\bsgamma}(p,\bsq). \]
Using Jensen's inequality, for any $1/\alpha < \lambda\leq 1$ we have
\begin{align*}
(B_{\bsgamma}(p,\bsq^*))^{\lambda} & \leq \frac{1}{(b^m-1)^s}\sum_{\bsq \in (G^*_{b,m})^s}(B_{\bsgamma}(p,\bsq))^{\lambda} \\
& \leq \frac{1}{(b^m-1)^s}\sum_{\bsq \in (G^*_{b,m})^s}\sum_{\emptyset \neq u\subseteq \{1,\ldots,s\}}\gamma_u^{\lambda}C_{\alpha}^{\lambda |u|}\sum_{\substack{\bsk_u\in P_u^{\perp}(p,\bsq)\\ \exists j\in u\colon b^m\nmid k_j}}b^{-\lambda \mu_{\alpha}(\bsk_u)} \\
& = \sum_{\emptyset \neq u\subseteq \{1,\ldots,s\}}\gamma_u^{\lambda}C_{\alpha}^{\lambda |u|}\sum_{\substack{\bsk_u\in \NN^{|u|}\\ \exists j\in u\colon b^m\nmid k_j}}b^{-\lambda \mu_{\alpha}(\bsk_u)}\\
& \qquad \times \frac{1}{(b^m-1)^{|u|}}\sum_{\substack{\bsq_u \in (G^*_{b,m})^{|u|}\\ \tr_m(\bsk_u)\cdot \bsq_u=0 \pmod p}}1.
\end{align*}
If there exists at least one component $k_j$ with $j\in u$ such that $b^m\nmid k_j$, the number of polynomials $\bsq_u\in (G^*_{b,m})^{|u|}$ which satisfies $\tr_m(\bsk_u)\cdot \bsq_u=0\pmod p$ is $(b^m-1)^{|u|-1}$. Thus we obtain
\begin{align*}
(B_{\bsgamma}(p,\bsq^*))^{\lambda} & \leq \frac{1}{b^m-1}\sum_{\emptyset \neq u\subseteq \{1,\ldots,s\}}\gamma_u^{\lambda}C_{\alpha}^{\lambda |u|}\sum_{\substack{\bsk_u\in \NN^{|u|}\\ \exists j\in u\colon b^m\nmid k_j}}b^{-\lambda \mu_{\alpha}(\bsk_u)} \\
& \leq \frac{1}{b^m-1}\sum_{\emptyset \neq u\subseteq \{1,\ldots,s\}}\gamma_u^{\lambda}C_{\alpha}^{\lambda |u|}E_{\alpha,\lambda}^{|u|}.
\end{align*}
This completes the proof.
\end{proof}

\subsection{Dependence of the upper bound on the dimension}\label{subsec:dependence}
Here we study the dependence of the worst-case error bound on the dimension.
For $1/\alpha < \lambda<1$, we write
\[ J_{s,\lambda,\bsgamma}=\left[ \sum_{\emptyset \neq u\subseteq \{1,\ldots,s\}}\gamma_u^{\lambda}C_{\alpha}^{\lambda|u|}E_{\alpha,\lambda}^{|u|}\right]^{1/\lambda}. \]
From Theorem~\ref{thm:bound_wrst_error} together with Theorem~\ref{thm:existence}, we have
\begin{align*}
& \sup_{\substack{f\in W_{s,\alpha,\bsgamma,q,\infty}\\ \|f\|_{s,\alpha,\bsgamma,q,\infty}\leq 1}}\left| \sum_{\tau=1}^{\alpha}a_{\tau}^{(\alpha)} I(f;P_{b^{m-\tau+1}}) -I(f)\right| \\
& \qquad \leq \sum_{\tau=1}^{\alpha}|a_{\tau}^{(\alpha)}| \left( \frac{J_{s,\lambda,\bsgamma}}{(b^{m-\tau+1}-1)^{1/\lambda}}+\frac{H_{s,\bsgamma,q,\infty}}{b^{\alpha(m-\tau+1)}}\right) \\
& \qquad \leq \sum_{\tau=1}^{\alpha}|a_{\tau}^{(\alpha)}| \frac{b^{\tau/\lambda}J_{s,\lambda,\bsgamma}+b^{\alpha(\tau-1)}H_{s,\bsgamma,q,\infty}}{(b^m-1)^{1/\lambda}} \leq \alpha |a_{1}^{(\alpha)}|\frac{b^{\alpha/\lambda}J_{s,\lambda,\bsgamma}+b^{\alpha(\alpha-1)}H_{s,\bsgamma,q,\infty}}{(b^m-1)^{1/\lambda}},
\end{align*}
for any $1/\alpha < \lambda<1$. Here we recall
\[ H_{s,\bsgamma,q,\infty} = \sum_{u\subseteq \{1,\ldots,s\}}\gamma_u(\alpha+1)^{|u|/q'}D_{\alpha}^{|u|}. \]
The dependence of the upper bound on the dimension can be stated as follows.

\begin{corollary}\label{cor:dependence}
Let $\alpha > 1$ be an integer and $N = b^m + b^{m-1} + \cdots + b^{m-\alpha+1}$ be the number of function evaluations used in the extrapolated polynomial lattice rule.
\begin{enumerate}
\item For general weights, assume that
\[ \lim_{s\to \infty}J_{s,\lambda,\bsgamma}<\infty\quad \text{and}\quad \lim_{s\to \infty}H_{s,\bsgamma,q,\infty} <\infty, \]
for some $1/\alpha<\lambda\leq 1$. Then the worst-case error for extrapolated polynomial lattice rules converges with order $\mathcal{O}(N^{-1/\lambda})$ with the constant bounded independently of the dimension.
\item For general weights, assume that there exists a positive real $q$ such that
\[ \limsup_{s\to \infty}\frac{J_{s,\lambda,\bsgamma}}{s^q}<\infty\quad \text{and}\quad \limsup_{s\to \infty}\frac{H_{s,\bsgamma,q,\infty}}{s^q} <\infty, \]
holds for some $1/\alpha<\lambda\leq 1$. Then the worst-case error bound for extrapolated polynomial lattice rules converges with order $\mathcal{O}(N^{-1/\lambda})$ with the constant depending polynomially on the dimension.
\item For product weights $\gamma_u=\prod_{j\in u}\gamma_j$, assume that
\[ \sum_{j=1}^{\infty}\gamma_j^{\lambda}<\infty, \]
for some $1/\alpha<\lambda\leq 1$. Then the worst-case error for extrapolated polynomial lattice rules converges with order $\mathcal{O}(N^{-1/\lambda})$ with the constant bounded independently of the dimension.
\item For product weights $\gamma_u=\prod_{j\in u}\gamma_j$, assume that
\[ \limsup_{s\to \infty}\frac{\sum_{j=1}^{s}\gamma_j^{\lambda}}{\log (s+1)}<\infty, \]
for some $1/\alpha<\lambda\leq 1$. Then the worst-case error bound for extrapolated polynomial lattice rules converges with order $\mathcal{O}(N^{-1/\lambda})$ with the constant  depending polynomially on the dimension.
\end{enumerate}
\end{corollary}

\begin{proof}
The results for general weights follows immediately.
The proof of the results for product weights can be also completed by following essentially the same argument as in \cite[Proof of Theorem~5.3]{DP07}.
\end{proof}

\begin{remark}
For product weights, good extrapolated polynomial lattice rules can be constructed as discussed in the next section. As can be seen from the error bound obtained in Theorem~\ref{thm:cbc}, if the same condition as Item~3 or 4 of Corollary~\ref{cor:dependence} holds, we also have exactly the same result for the dependence of the worst-case error bound on the dimension.
\end{remark}
\section{Component-by-component construction}\label{sec:cbc}

\subsection{Convergence analysis}
Here we only consider the case of product weights and prove that the CBC construction algorithm can find a good polynomial lattice rule which achieves the almost optimal order bound on the criterion $B_{\bsgamma}(p,\bsq)$. Remark~\ref{rem_prod} below points out the challenge in generalizing the result to general weights.

The CBC construction algorithm proceeds as follows:
\begin{algorithm}
\label{alg:cbc}
For $m,s\in \NN$, $\alpha\geq 2$ and $\bsgamma=(\gamma_j)_{j\in \NN}$.
\begin{enumerate}
\item Choose an irreducible polynomial $p\in \FF_b[x]$ with $\deg(p)=m$.
\item Set $q_1^*=1$.
\item For $2\leq d \leq s$, find $q_d^*\in G^*_{b,m}$ which minimizes $$B_{\bsgamma}(p,(q^*_1,\ldots,q^*_{d-1},q_d))=\sum_{\emptyset \neq u\subseteq \{1,\ldots,d \}}\gamma_uC_{\alpha}^{|u|}\sum_{\substack{\bsk_u\in P_u^{\perp}(p,(q^*_1,\ldots,q^*_{d-1},q_d))\\ \exists j\in u\colon b^m\nmid k_j}}b^{-\mu_{\alpha}(\bsk_u)}$$ as a function of $q_d$.
\end{enumerate}
\end{algorithm}

\noindent In Section~\ref{sec_fast} we simplify the formula for $B_{\bsgamma}(p,(q^*_1,\ldots,q^*_{d-1},q_d))$ to obtain a criterion which can be computed efficiently.

\begin{theorem}\label{thm:cbc}
Let $p\in \FF_b[x]$ with $\deg(p)=m$ and $\bsq_s^*=(q_1^*,\ldots,q_s^*)\in (G^*_{b,m})^s$ be found by Algorithm~\ref{alg:cbc}. Then for $1\leq d \leq s$ we have
\begin{align*}
B_{\bsgamma}(p,\bsq_d^*)\leq \frac{1}{(b^m-1)^{1/\lambda}}\prod_{j=1}^{d}\left[ 1+ \gamma_j^{\lambda} C_{\alpha}^{\lambda}E_{\alpha,\lambda}\right]^{1/\lambda}
\end{align*}
holds for any $1/\alpha<\lambda\leq 1$.
\end{theorem}

\begin{proof}
Without loss of generality, we assume that the modulus $p$ is monic.
We prove the theorem by induction on $d$.
First let $d=1$.
Since we assume $q_1^*=1$, the dual polynomial lattice is given by
\[ P^{\perp}(p,1) = \{k\in \NN_0\colon \tr_m(k)=0\pmod p\} = \{k\in \NN_0\colon b^m\mid k\}. \]
Thus we have
\[ B_{\bsgamma}(p,1)= C_{\alpha}\gamma_1 \sum_{\substack{k\in P^{\perp}(p,1)\setminus \{0\}\\ b^m\nmid k}}b^{-\mu_{\alpha}(k)}=0 \leq  \frac{1}{(b^m-1)^{1/\lambda}}\left(1+\gamma^{\lambda}_1C^{\lambda}_{\alpha}E_{\alpha,\lambda}\right)^{1/\lambda}, \]
for any $1/\alpha<\lambda\leq 1$.

Next suppose that we have already found the first $d-1$ components of the generating vector $\bsq_{d-1}^*=(q^*_1,\ldots,q^*_{d-1})\in (G^*_{b,m})^{d-1}$ such that
\begin{align*}
B_{\bsgamma}(p,\bsq_{d-1}^*)\leq \frac{1}{(b^m-1)^{1/\lambda}}\prod_{j=1}^{d-1}\left[ 1+ \gamma_j^{\lambda} C_{\alpha}^{\lambda}E_{\alpha,\lambda}\right]^{1/\lambda}
\end{align*}
holds for any $1/\alpha<\lambda\leq 1$.
Putting $\bsq_d=(\bsq_{d-1}^*,q_d)$ with $q_d \in G^*_{b,m}$ we have
\begin{align}
B_{\bsgamma}(p,\bsq_d) & = \sum_{\emptyset \neq u\subseteq \{1,\ldots,d-1\}}\gamma_u C_{\alpha}^{|u|}\sum_{\substack{\bsk_u\in P_u^{\perp}(p,\bsq_d)\\ \exists j\in u\colon b^m\nmid k_j}}b^{-\mu_{\alpha}(\bsk_u)} \nonumber \\
& \qquad + \sum_{\emptyset \neq u\subseteq \{1,\ldots,d-1\}}\gamma_{u\cup\{d\}} C_{\alpha}^{|u|+1}\sum_{\substack{\bsk_{u\cup\{d\}}\in P_{u\cup\{d\}}^{\perp}(p,\bsq_d)\\ \exists j\in u\colon b^m\nmid k_j\\ b^m\mid k_d}}b^{-\mu_{\alpha}(\bsk_{u\cup\{d\}})} \nonumber \\
& \qquad + \sum_{u\subseteq \{1,\ldots,d-1\}}\gamma_{u\cup\{d\}}C_{\alpha}^{|u|+1}\sum_{\substack{\bsk_{u\cup\{d\}}\in P_{u\cup\{d\}}^{\perp}(p,\bsq_d)\\ b^m\nmid k_d}}b^{-\mu_{\alpha}(\bsk_{u\cup\{d\}})} \nonumber \\
& = B_{\bsgamma}(p,\bsq^*_{d-1}) + \sum_{\emptyset \neq u\subseteq \{1,\ldots,d-1\}}\gamma_{u\cup\{d\}}C_{\alpha}^{|u|+1}\sum_{\substack{\bsk_u\in P_u^{\perp}(p,\bsq_{d-1}^*)\\ \exists j\in u\colon b^m\nmid k_j}}\sum_{\substack{k_d\in \NN\\ b^m\mid k_d}}b^{-\mu_{\alpha}(\bsk_u,k_d)} \nonumber \\
& \qquad + \sum_{u\subseteq \{1,\ldots,d-1\}}\gamma_{u\cup\{d\}}C_{\alpha}^{|u|+1}\sum_{\substack{\bsk_{u\cup\{d\}}\in P_{u\cup\{d\}}^{\perp}(p,\bsq_d)\\ b^m\nmid k_d}}b^{-\mu_{\alpha}(\bsk_{u\cup\{d\}})} \nonumber \\
& = B_{\bsgamma}(p,\bsq^*_{d-1}) \left(1+ \gamma_d C_{\alpha}\sum_{\substack{k_d\in \NN\\ b^m\mid k_d}}b^{-\mu_{\alpha}(k_d)}\right)  \nonumber \\
& \qquad + \sum_{u\subseteq \{1,\ldots,d-1\}}\gamma_{u\cup\{d\}}C_{\alpha}^{|u|+1}\sum_{\substack{\bsk_{u\cup\{d\}}\in P_{u\cup\{d\}}^{\perp}(p,\bsq_d)\\ b^m\nmid k_d}}b^{-\mu_{\alpha}(\bsk_{u\cup\{d\}})} , \label{eq:decomp_error}
\end{align}
where the second equality stems from the fact that since $b^m\mid k_d$, we have $\tr_m(k_d)=0$ and thus $\tr_m(\bsk_{u\cup \{d\}})\cdot (\bsq_u^*, q_d)=\tr_m(\bsk_u)\cdot \bsq_u^*$, which yields 
\[ \{\bsk_{u\cup \{d\}}\in P_{u\cup \{d\}}^{\perp}(p,\bsq_d)\colon b^m\mid k_d\}= \{(\bsk_u,k_d)\in \NN^{|u|+1}\colon \bsk_u\in P_u^{\perp}(p,\bsq_{d-1}^*), b^m \mid k_d\}. \]
It is clear that the first term of \eqref{eq:decomp_error} does not depend on the choice of $q_d$.
Thus denoting the second term of \eqref{eq:decomp_error} by
\[ \psi_{p,\bsq_{d-1}^*}(q_d):=  \sum_{u\subseteq \{1,\ldots,d-1\}}\gamma_{u\cup\{d\}}C_{\alpha}^{|u|+1}\sum_{\substack{\bsk_{u\cup\{d\}}\in P_{u\cup\{d\}}^{\perp}(p,\bsq_d)\\ b^m\nmid k_d}}b^{-\mu_{\alpha}(\bsk_{u\cup\{d\}})}, \]
we have
\[ q_d^*=\arg\min_{q_d\in G^*_{b,m}} B_{\bsgamma}(p,\bsq_d) = \arg\min_{q_d\in G^*_{b,m}} \psi_{p,\bsq_{d-1}^*}(q_d) . \]

Using Jensen's inequality, as long as $1/\alpha<\lambda\leq 1$, we have
\begin{align*}
& (\psi_{p,\bsq_{d-1}^*}(q_d^*))^{\lambda} \\
& \leq \frac{1}{b^m-1}\sum_{q_d\in G^*_{b,m}}(\psi_{p,\bsq_{d-1}^*}(q_d))^{\lambda} \\
& \leq \frac{1}{b^m-1}\sum_{q_d\in G^*_{b,m}}\sum_{u\subseteq \{1,\ldots,d-1\}}\gamma^{\lambda}_{u\cup\{d\}}C_{\alpha}^{\lambda(|u|+1)}\sum_{\substack{\bsk_{u\cup\{d\}}\in P_{u\cup\{d\}}^{\perp}(p,\bsq_d)\\ b^m\nmid k_d}}b^{-\lambda\mu_{\alpha}(\bsk_{u\cup\{d\}})} \\
& = \frac{1}{b^m-1}\sum_{u\subseteq \{1,\ldots,d-1\}}\gamma^{\lambda}_{u\cup\{d\}}C_{\alpha}^{\lambda(|u|+1)}\sum_{\substack{\bsk_{u\cup\{d\}}\in \NN^{|u|+1}\\ b^m\nmid k_d}}b^{-\lambda\mu_{\alpha}(\bsk_{u\cup\{d\}})} \\
& \qquad \times \sum_{\substack{q_d\in G^*_{b,m}\\ \tr_m(\bsk_u)\cdot \bsq_u^{*}+\tr_m(k_d)q_d = 0\pmod p}}1.
\end{align*}
Since $b^m\nmid k_d$, we have $\tr_m(k_d)\neq 0$.
For $\bsk_u\in P_u^{\perp}(p,\bsq_{d-1}^*)$, it follows from the definition of the dual polynomial lattice that $\tr_m(\bsk_u)\cdot \bsq_u^{*}=0\pmod p$, and thus there is no polynomial $q_d\in G^*_{b,m}$ such that the condition $\tr_m(k_d)q_d = 0\pmod p$ is satisfied.
For $\bsk_u\notin P_u^{\perp}(p,\bsq_{d-1}^*)$, there exists exactly one $q_d\in G^*_{b,m}$ such that $\tr_m(k_d)q_d = -\tr_m(\bsk_u)\cdot \bsq_u^{*} \pmod p$. From these facts and Lemma~\ref{lem:sum_mu_alpha}, we obtain
\begin{align*}
(\psi_{p,\bsq_{d-1}^*}(q_d^*))^{\lambda} & \leq \frac{1}{b^m-1}\sum_{u\subseteq \{1,\ldots,d-1\}}\gamma^{\lambda}_{u\cup\{d\}}C_{\alpha}^{\lambda(|u|+1)}\sum_{\substack{\bsk_u\in \NN^{|u|}\\ \bsk_u\notin P_u^{\perp}(p,\bsq_{d-1}^*)}}\sum_{\substack{k_d\in \NN\\  b^m\nmid k_d}}b^{-\lambda\mu_{\alpha}(\bsk_u,k_d)} \\
& \leq \frac{1}{b^m-1}\sum_{u\subseteq \{1,\ldots,d-1\}}\gamma^{\lambda}_{u\cup\{d\}}C_{\alpha}^{\lambda(|u|+1)}\sum_{\bsk_u\in \NN^{|u|}}b^{-\lambda\mu_{\alpha}(\bsk_u)}\sum_{\substack{k_d\in \NN\\  b^m\nmid k_d}}b^{-\lambda\mu_{\alpha}(k_d)} \\
& =  \frac{1}{b^m-1}\prod_{j=1}^{d-1}\left[ 1+ \gamma_j^{\lambda} C_{\alpha}^{\lambda}E_{\alpha,\lambda}\right]\cdot \gamma_d^{\lambda} C_{\alpha}^{\lambda}\sum_{\substack{k_d\in \NN\\  b^m\nmid k_d}}b^{-\lambda\mu_{\alpha}(k_d)}.
\end{align*}
Finally by applying Jensen's inequality to \eqref{eq:decomp_error} and using Lemma~\ref{lem:sum_mu_alpha}, we have
\begin{align*}
(B_{\bsgamma}(p,\bsq^*_d))^{\lambda} & \leq (B_{\bsgamma}(p,\bsq^*_{d-1}))^{\lambda}\left(1+ \gamma_d^{\lambda}C^{\lambda}_{\alpha}\sum_{\substack{k_d\in \NN\\ b^m\mid k_d}}b^{-\lambda\mu_{\alpha}(k_d)}\right)  \\
& \qquad +  \frac{1}{b^m-1}\prod_{j=1}^{d-1}\left[ 1+ \gamma_j^{\lambda} C_{\alpha}^{\lambda}E_{\alpha,\lambda}\right]\cdot \gamma_d^{\lambda} C_{\alpha}^{\lambda}\sum_{\substack{k_d\in \NN\\  b^m\nmid k_d}}b^{-\lambda\mu_{\alpha}(k_d)} \\
& \leq \frac{1}{b^m-1}\prod_{j=1}^{d-1}\left[ 1+ \gamma_j^{\lambda} C_{\alpha}^{\lambda}E_{\alpha,\lambda}\right] \cdot \left[ 1+\gamma_d^{\lambda}C^{\lambda}_{\alpha}\sum_{k_d\in \NN}b^{-\lambda\mu_{\alpha}(k_d)}\right] \\
& = \frac{1}{b^m-1}\prod_{j=1}^{d}\left[ 1+ \gamma_j^{\lambda} C_{\alpha}^{\lambda}E_{\alpha,\lambda}\right] .
\end{align*}
This completes the proof.
\end{proof}

\begin{remark}\label{rem_prod}
In the above proof, we use the property of product weights to obtain the equality \eqref{eq:decomp_error}.
In fact, this is a crucial step to get the almost optimal order upper bound on $B_{\bsgamma}(p,\bsq)$.
Thus it is an open question whether a similar proof goes through for general weights.
\end{remark}
\subsection{Fast construction algorithm}\label{sec_fast}
In the convergence analysis above, we used the criterion $B_{\bsgamma}(p,\bsq_d)$.
However, since the quantity
\[ \sum_{\emptyset \neq u\subseteq \{1,\ldots,d\}}\gamma_uC_{\alpha}^{|u|}\sum_{\substack{\bsk_u\in P_u^{\perp}(p,\bsq_d)\\ \forall j\in u\colon b^m\mid k_j}}b^{-\mu_{\alpha}(\bsk_u)} = \sum_{\emptyset \neq u\subseteq \{1,\ldots,d\}}\gamma_uC_{\alpha}^{|u|}\sum_{\bsk_u\in \NN^{|u|}}b^{-\mu_{\alpha}(b^m\bsk_u)} \]
does not depend on the choice of generating vector $\bsq_d$, we can add this quantity to the criterion $B_{\bsgamma}(p,\bsq_d)$ to get another criterion
\begin{align*}
\tilde{B}_{\bsgamma}(p,\bsq_d) & = B_{\bsgamma}(p,\bsq_d) +  \sum_{\emptyset \neq u\subseteq \{1,\ldots,d\}}\gamma_u C_{\alpha}^{|u|}\sum_{\substack{\bsk_u\in P_u^{\perp}(p,\bsq_d)\\ \forall j\in u\colon b^m\mid k_j}}b^{-\mu_{\alpha}(\bsk_u)} \\
& = \sum_{\emptyset \neq u\subseteq \{1,\ldots,d\}}\gamma_u C_{\alpha}^{|u|}\sum_{\bsk_u\in P_u^{\perp}(p,\bsq_d)}b^{-\mu_{\alpha}(\bsk_u)} \\
& = -1+\frac{1}{b^m}\sum_{n=0}^{b^m-1}\sum_{u\subseteq \{1,\ldots,d\}}\gamma_u C_{\alpha}^{|u|}\sum_{\bsk_u\in \NN^{|u|}}b^{-\mu_{\alpha}(\bsk_u)}\wal_{(\bsk_u,\bszero)}(\bsx_n) \\
& = -1+\frac{1}{b^m}\sum_{n=0}^{b^m-1}\prod_{j=1}^{d}\left[ 1+ \gamma_j C_{\alpha} w_{\alpha}\left(v_m\left( \frac{nq_j}{p}\right) \right)\right],
\end{align*}
where we used Lemma~\ref{lem:character} in the third equality, and the function $w_{\alpha}\colon [0,1)\to \RR$ is defined by
\[ w_{\alpha}(x) = \sum_{k=1}^{\infty}b^{-\mu_{\alpha}(k)}\wal_k(x). \]
As shown in \cite[Theorem~2]{BDLNP12}, one can compute $w_{\alpha}$ efficiently when $x$ is a $b$-adic rational.
More precisely, if $x$ is of the form $a/b^m$ for $m\in \NN$ and $0\leq a< b^m$, $w_{\alpha}(x)$ can be computed in at most $O(\alpha m)$ operations.
Furthermore, in case of $b=2$, we have explicit formulas for $w_2$ and $w_3$, see \cite[Corollary~1]{BDLNP12}. 

In what follows, we show how one can use the fast CBC construction algorithm to find suitable polynomials $q_1^*,\ldots,q_s^*\in G^*_{b,m}$ by employing $\tilde{B}_{\bsgamma}(p,\bsq)$ as a quality measure.
Assume that $q^*_1=1,q^*_2,\ldots,q^*_{d-1}$ are already found.
Let
\[ P_{n,d-1}=\prod_{j=1}^{d-1}\left[ 1+ \gamma_j C_{\alpha}  w_{\alpha}\left(v_m\left( \frac{nq^*_j}{p}\right) \right)\right], \]
for $0\leq n<b^m$. Note that we have
\[ P_{0,d-1}=\prod_{j=1}^{d-1}\left[ 1+  \gamma_j C_{\alpha} w_{\alpha}\left(0 \right)\right] , \]
regardless of the choice $q^*_1,q^*_2,\ldots,q^*_{d-1}$.
Now the criterion $\tilde{B}_{\bsgamma}(p,\bsq_d)$ is given by
\begin{align*}
\tilde{B}_{\bsgamma}(p,\bsq_d) & = -1+\frac{1}{b^m}\sum_{n=0}^{b^m-1}P_{n,d-1}\left[ 1+ \gamma_d C_{\alpha} w_{\alpha}\left(v_m\left( \frac{nq_d}{p}\right) \right)\right] \\
& = -1+\frac{P_{0,d}}{b^m} + \frac{1}{b^m}\sum_{n=1}^{b^m-1}P_{n,d-1}\left[ 1+ \gamma_d C_{\alpha}  w_{\alpha}\left(v_m\left( \frac{nq_d}{p}\right) \right)\right] \\
& = -1+\frac{P_{0,d}}{b^m} + \frac{1}{b^m}\sum_{n=1}^{b^m-1}P_{n,d-1}+ \frac{\gamma_d C_{\alpha}}{b^m}\sum_{n=1}^{b^m-1}P_{n,d-1} w_{\alpha}\left(v_m\left( \frac{nq_d}{p}\right) \right).
\end{align*}
Thus it is obvious that the CBC algorithm finds a component $q_d^*$ which minimizes the last sum.

Since the modulus $p$ is assumed to be irreducible, there exists a primitive polynomial $g\in \FF_b[x]/p$ for which we have $\{g^0=g^{b^m-1}=1, g^{1},\ldots,g^{b^m-2}\}=(\FF_b[x]/p)\setminus \{0\}$, and then the last sum for a polynomial $q_d=g^{z}$ with $1\leq z\leq b^m-1$ is equivalent to
\begin{align*}
\sum_{n=1}^{b^m-1}P_{n,d-1} w_{\alpha}\left(v_m\left( \frac{nq_d}{p}\right) \right) = \sum_{n=1}^{b^m-1} P_{g^{-n},d-1} w_{\alpha}\left(v_m\left( \frac{g^{z-n}}{p}\right) \right) =: \eta_z,
\end{align*}
where we note that the subscript $g^{-n}$ appearing in $P_{g^{-n},d-1}$ is identified with the integer in $\{1,\ldots,b^m-1\}$.
We define the circulant matrix
\[ A = \omega_{\alpha}\left( v_m\left( \frac{g^{z-n}}{p}\right)\right)_{1\leq z,n\leq b^m-1},\]
and compute
\[ (\eta_1,\ldots,\eta_{b^m-1})^{\top} = A \cdot (P_{g^{-1},\tau-1},P_{g^{-2},\tau-1},\ldots,P_{g^{-b^m+1},d-1})^{\top}.\]
Let $z_0$ be an integer such that $\eta_{z_0}\leq \eta_{z}$ holds for any $1\leq z\leq b^m-1$. Then we set $q_d^*=g^{z_0}$.
Since the matrix $A$ is circulant, the matrix-vector multiplication above can be done by using the fast Fourier transform  in $O(mb^m)$ arithmetic operations with $O(b^m)$ memory space for $P_{n,d-1}$, see \cite{NC06a,NC06b}.
Therefore, we can compute the vector $(\eta_1,\ldots,\eta_{b^m-1})$ in a fast way.
After finding $q_d^*=g^{z_0}$, each $P_{n,d-1}$ is updated simply by
\[ P_{g^{-n},d} = P_{g^{-n},d-1} \left( 1+ \gamma_d C_{\alpha} w_{\alpha}\left(v_m\left( \frac{g^{z_0-n}}{p}\right) \right)\right). \]

Since each element of the circulant matrix $A$ can be calculated in at most $O(\alpha m)$ arithmetic operations, calculating one row (or one column) of $A$ requires $O(\alpha mb^m)$ arithmetic operations as the first step of the CBC algorithm.
Then the CBC algorithm proceeds in an inductive way as described above, yielding $O((s+\alpha)mb^m)$ arithmetic operations with $O(b^m)$ memory space for finding the generating vector $\bsq^*\in (G^*_{b,m})^s$.
Further, for an extrapolated polynomial lattice rule, we need to construct polynomial lattice rules with $\alpha$ consecutive sizes of nodes, $b^{m-\alpha+1},\ldots,b^{m}$, implying that the total number of points is $N=b^{m-\alpha+1}+\cdots+b^{m}$.
The obvious inequality
\[ \sum_{\tau=1}^{\alpha}(s+\alpha)(m-\tau+1)b^{m-\tau+1} \leq (s+\alpha)m N \leq (s+\alpha)N\log_b N\] 
shows that the total construction cost is of $O((s+\alpha)N \log N)$ together with $O(N)$ memory space, which improves the currently known result for an interlaced polynomial lattice rule that requires $O(s\alpha N\log N)$ arithmetic operations with $O(N)$ memory space \cite{G15,GD15}.

\section{Numerical experiments}\label{sec:experiment}
As a low-dimensional problem, let us consider a simple bi-variate test function
\[ f(x,y)=\frac{ye^{xy}}{e-2}, \]
whose exact value of $I(f)$ equals 1. This function has been often used in the literature, see for instance \cite[Chapter~8]{SJbook}. We approximate $I(f)$ by using extrapolated polynomial lattice rules over $\FF_2$ and also by using interlaced polynomial lattice rules over $\FF_2$ for comparison. Here extrapolated polynomial lattice rules are constructed by the fast CBC algorithm as described in Section~\ref{sec_fast} with the constant $C_{\alpha}=1$, which is justified as mentioned in Remark~\ref{rem:walsh_bound}, whereas interlaced polynomial lattice rules are constructed by the fast CBC algorithm based on a computable quality criterion given in \cite[Corollary~3]{G15}. For both the rules, we set $\gamma_1=\gamma_2=1$ within the CBC algorithm.

\begin{figure}[!b]
\centering
\includegraphics[width=0.49\textwidth]{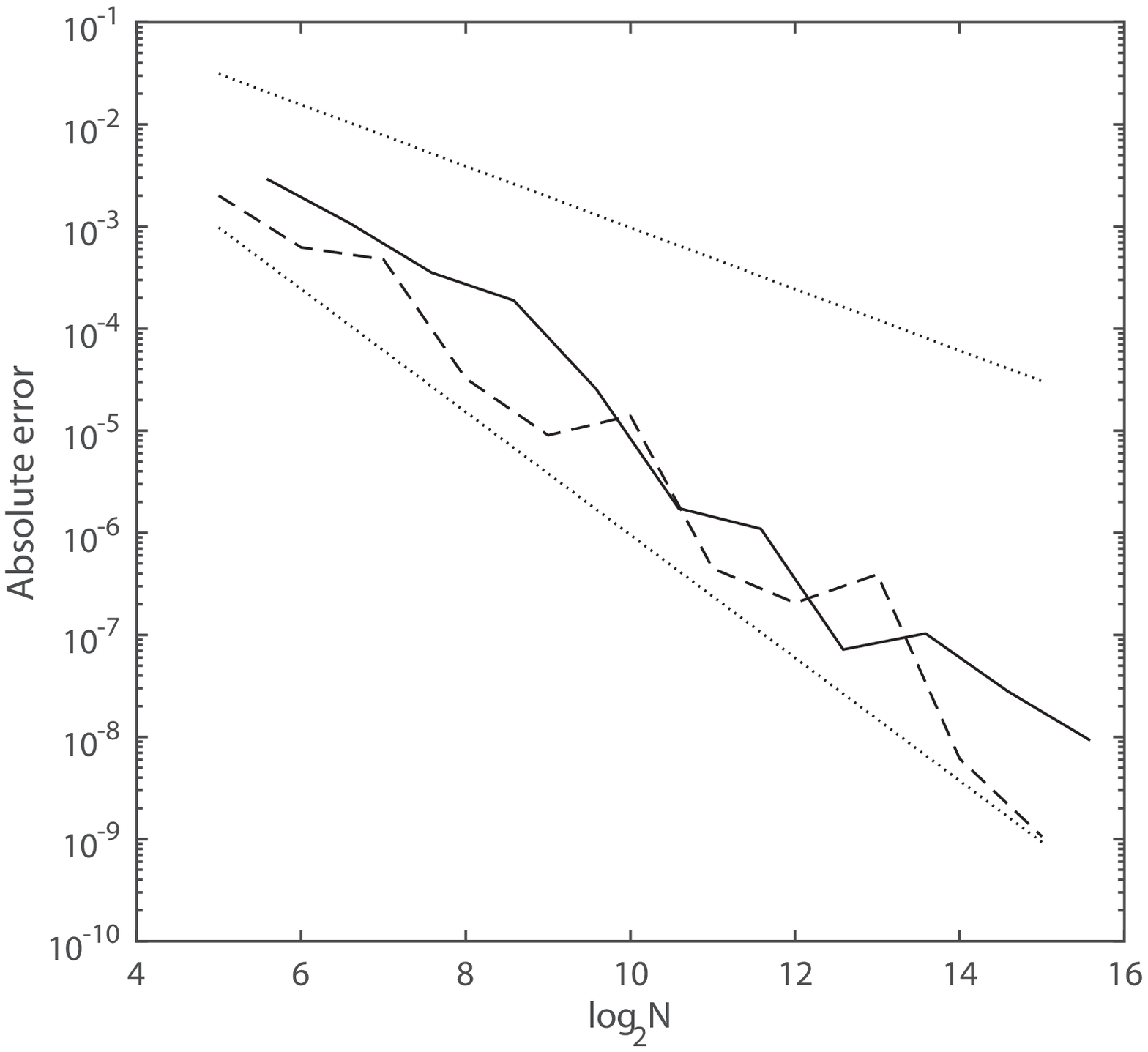}
\includegraphics[width=0.49\textwidth]{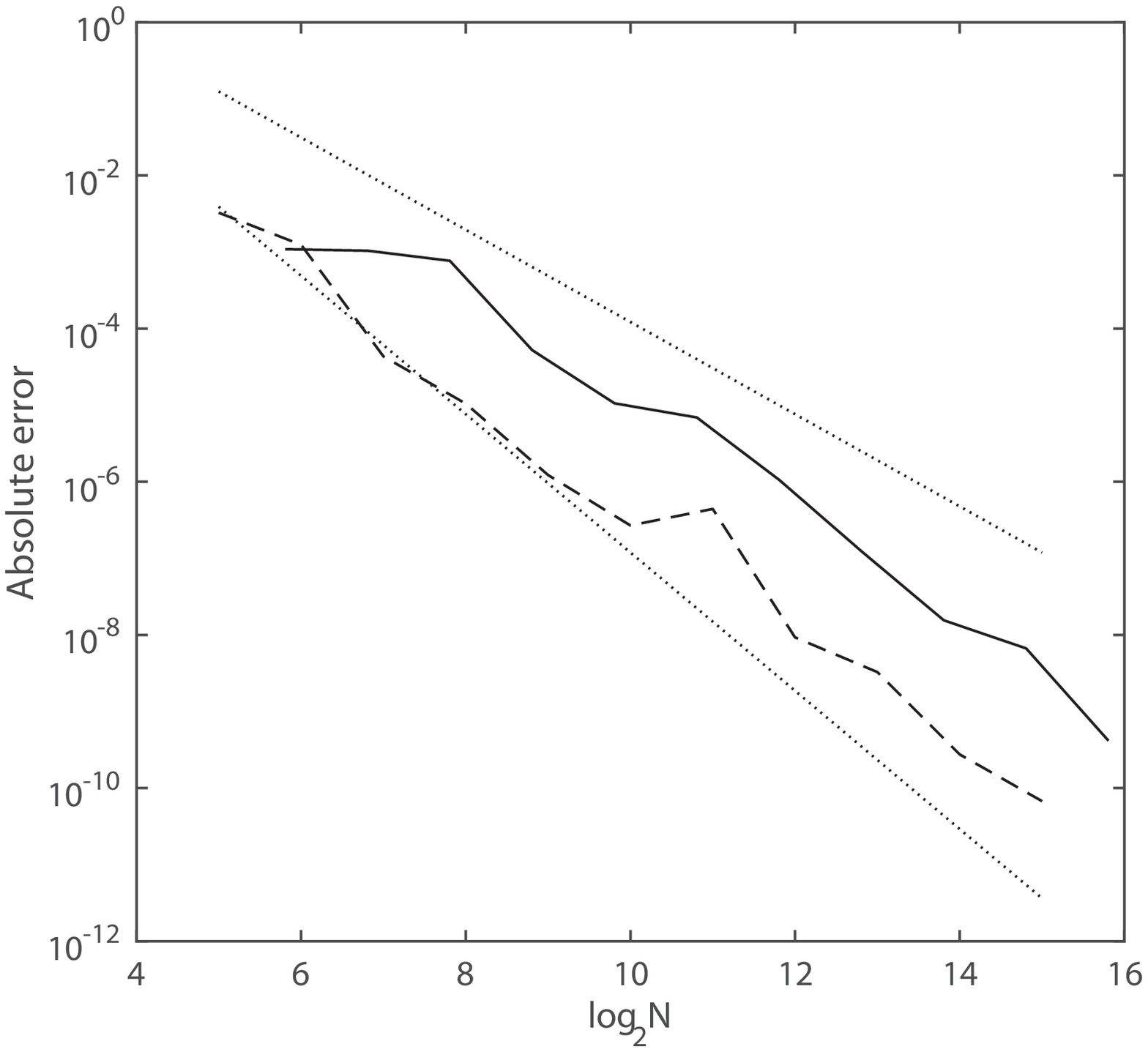}
\caption{The results for $f(x,y)=ye^{xy}/(e-2)$ by using extrapolated polynomial lattice rules (solid) and interlaced polynomial lattice rules (dashed) with $\alpha=2$ (left) and $\alpha=3$ (right).}
\label{fig:test0}
\end{figure}
Figure~\ref{fig:test0} shows the results for the cases $\alpha=2$ (left) and $\alpha=3$ (right). The absolute integration errors as functions of $\log_2 N$ are shown in each graph. The solid lines denote the results for extrapolated polynomial lattice rules and the dashed lines for interlaced polynomial lattice rules. For reference, the dotted lines correspond to $O(N^{-1})$ and $O(N^{-2})$ convergences for $\alpha=2$, and to $O(N^{-2})$ and $O(N^{-3})$ convergences for $\alpha=3$. For the case $\alpha=2$, both the rules perform comparably and achieve approximately the desired rate of the error convergence $O(N^{-2})$. For the case $\alpha=3$, although interlaced polynomial lattice rules outperform extrapolated polynomial lattice rules, we see that the rate of the error convergence for extrapolated polynomial lattice rules asymptotically improves towards the expected $O(N^{-3})$, which supports our theoretical funding.

Next let us consider the following high-dimensional test integrands 
\begin{align*}
 f_1(\bsx) & = \prod_{j=1}^{s}\left[ 1+\gamma_j\left(x_j^{c_1}-\frac{1}{1+c_1}\right)\right] ,\\
 f_2(\bsx) & = \prod_{j=1}^{s}\left[ 1+\frac{\gamma_j}{1+\gamma_j x_j^{c_2}} \right] ,
\end{align*}
for positive constants $c_1,c_2>0$.
Note that the exact values of the integrals for $f_1$ and for $f_2$ with the special cases $c_2=1$ and $c_2=2$ are known. We put $s=100$ and $\gamma_j=j^{-2}$. We construct both extrapolated polynomial lattice rules and interlaced polynomial lattice rules by using the fast CBC algorithm with the same choice of the weights $\gamma_j=j^{-2}$. Note that, in our experiments, we do not observe the phenomenon that the same elements of the generating vector  repeat as pointed out in \cite{GS16}.

\begin{figure}[!p]
\centering
\includegraphics[width=0.49\textwidth]{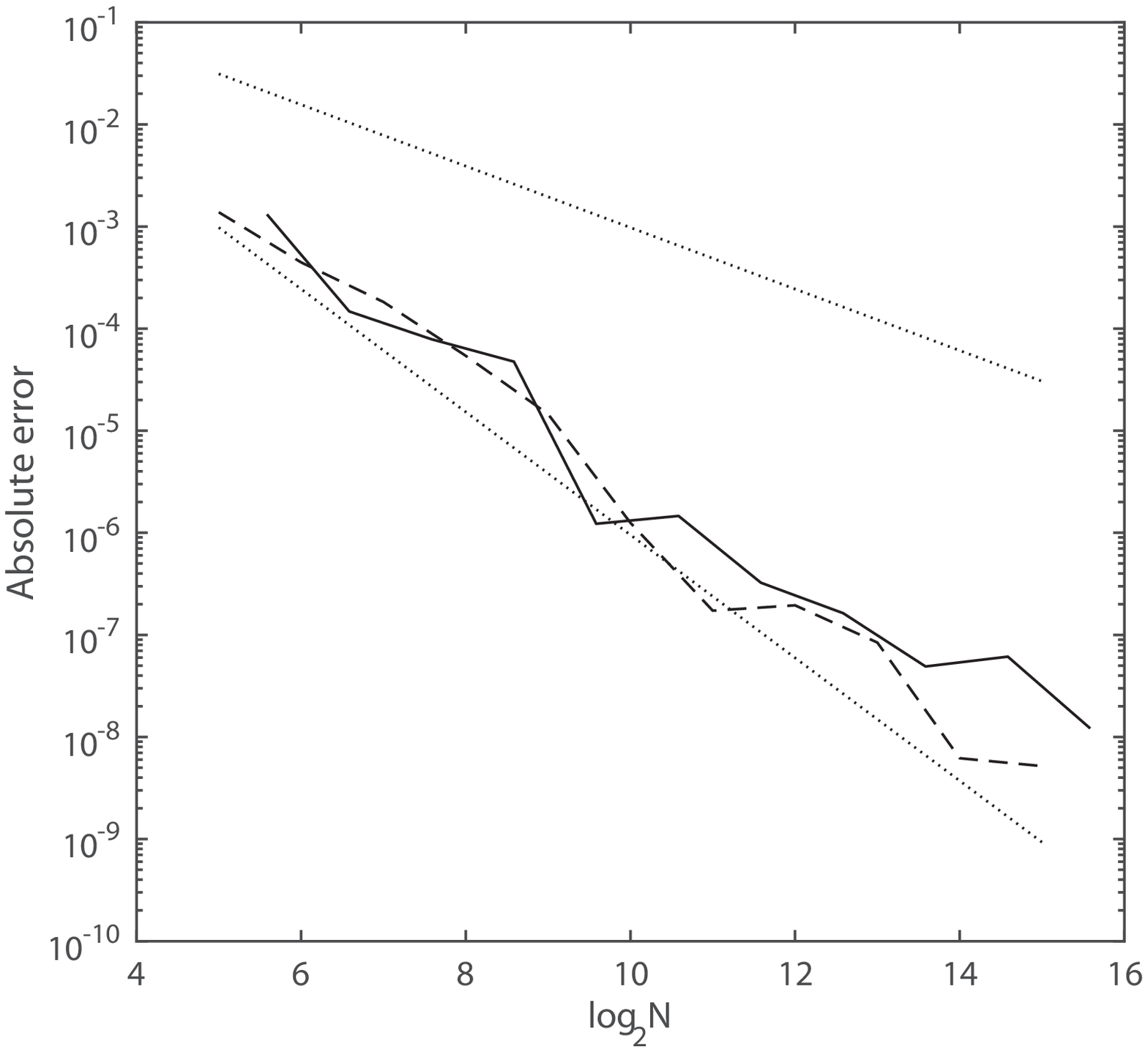}
\includegraphics[width=0.49\textwidth]{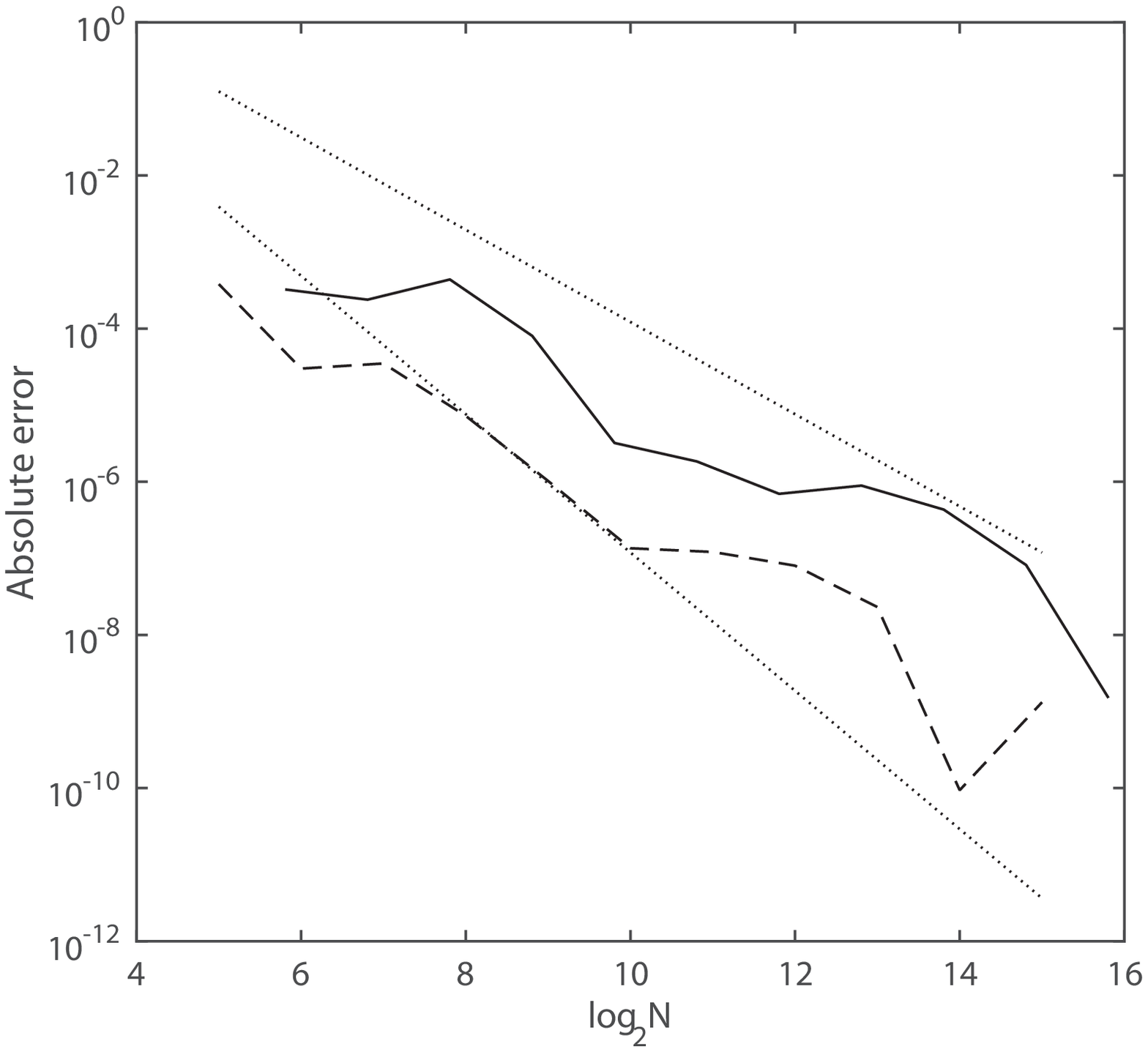}\\
\includegraphics[width=0.49\textwidth]{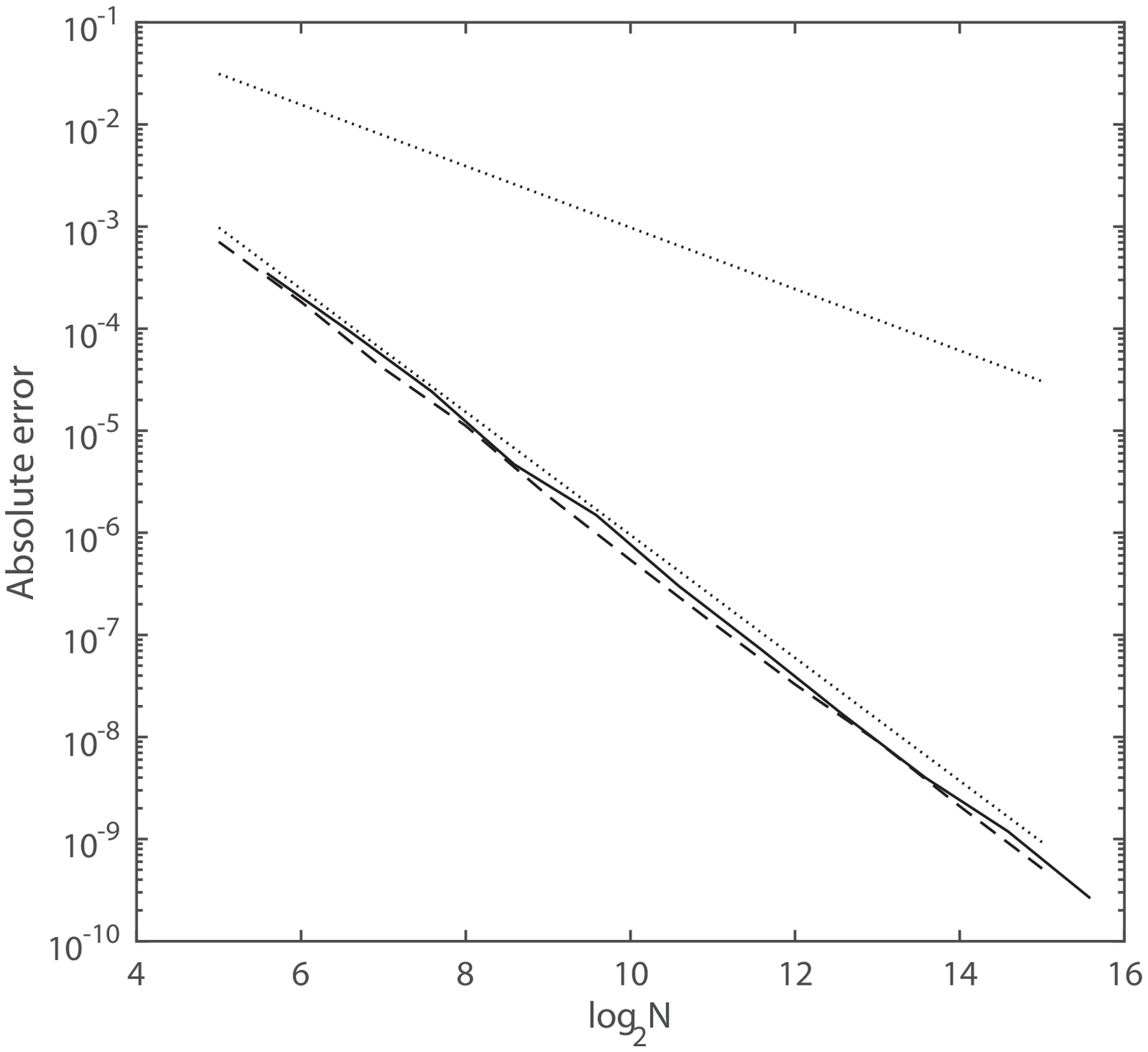}
\includegraphics[width=0.49\textwidth]{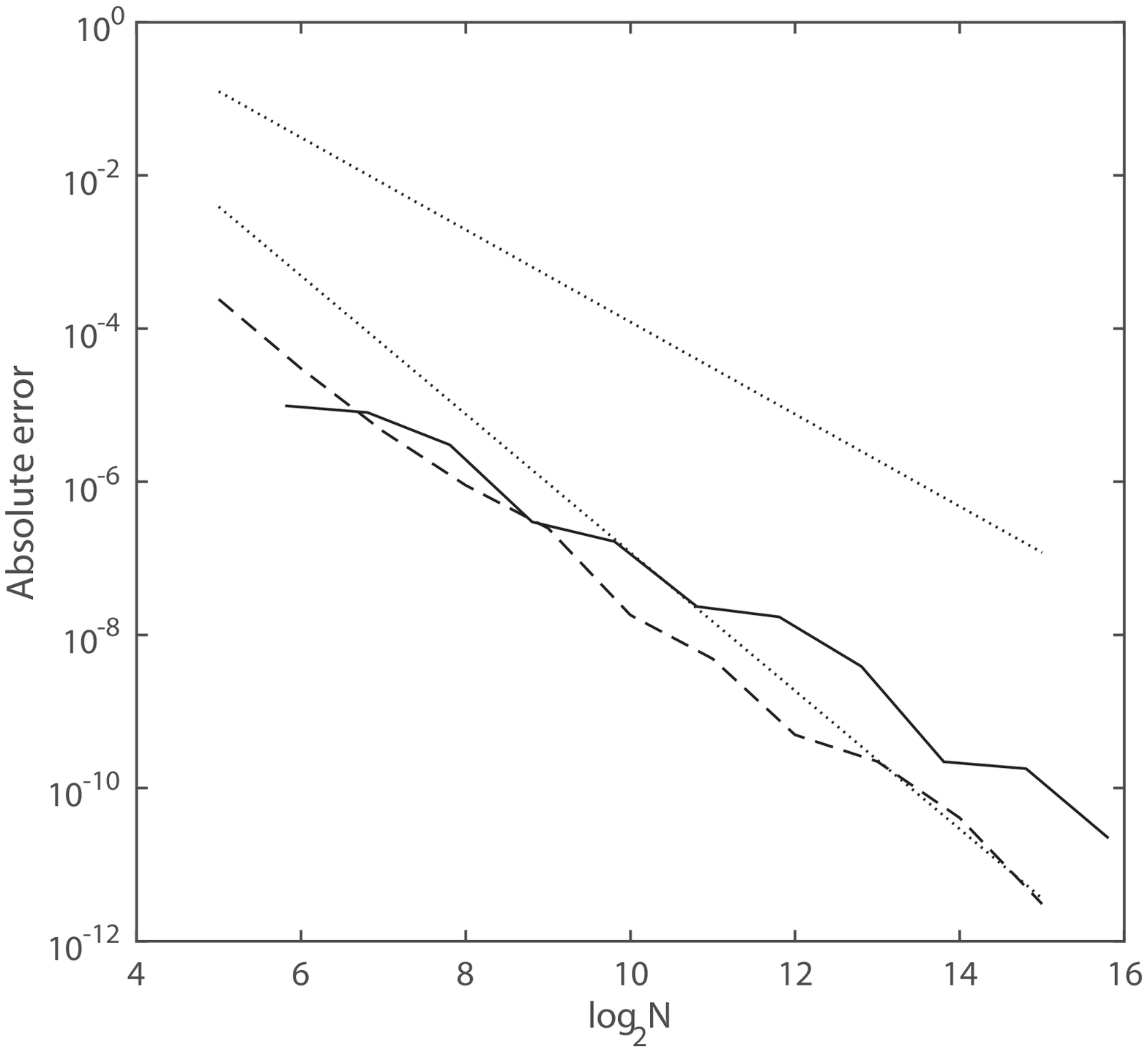}\\
\includegraphics[width=0.49\textwidth]{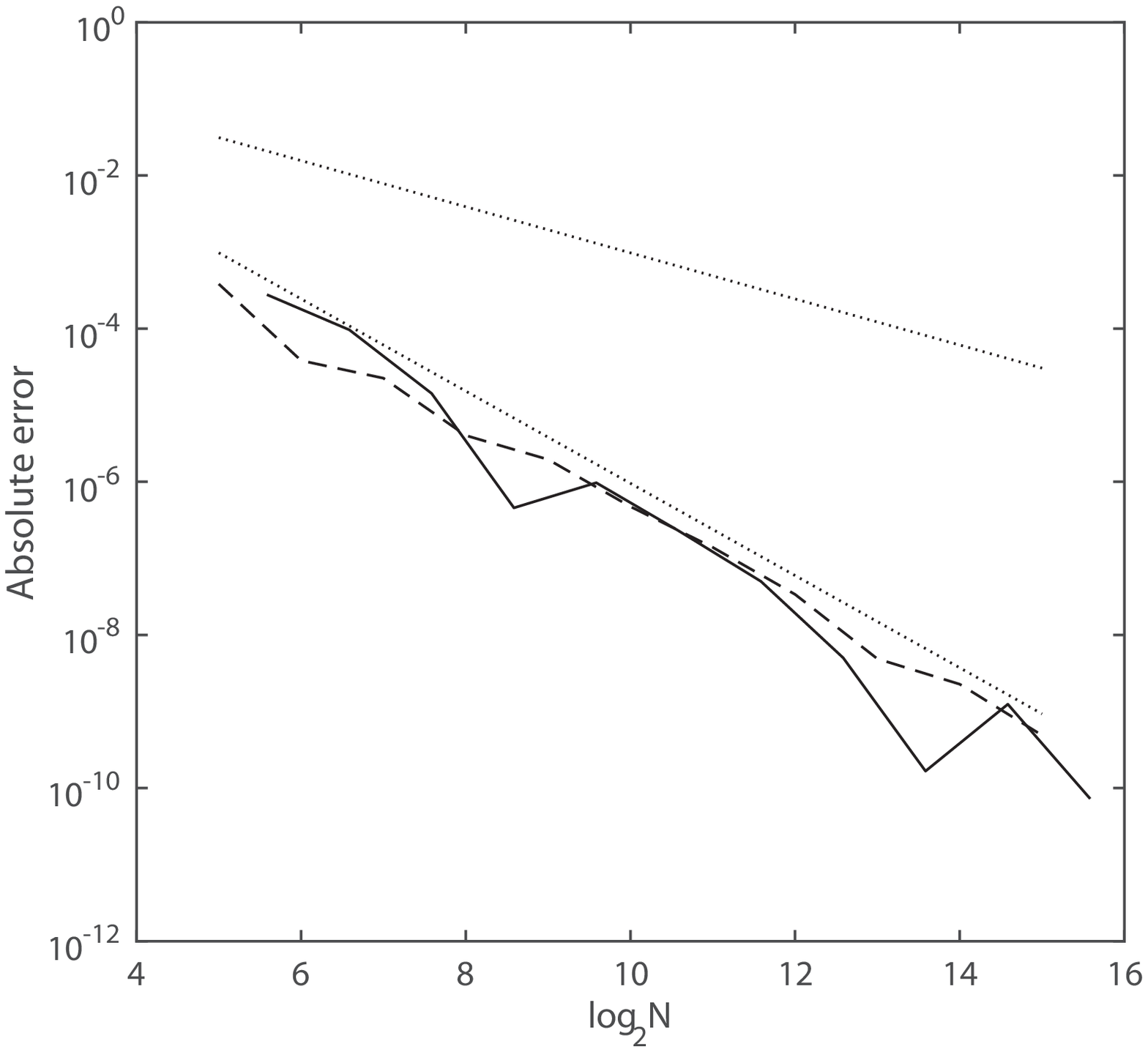}
\includegraphics[width=0.49\textwidth]{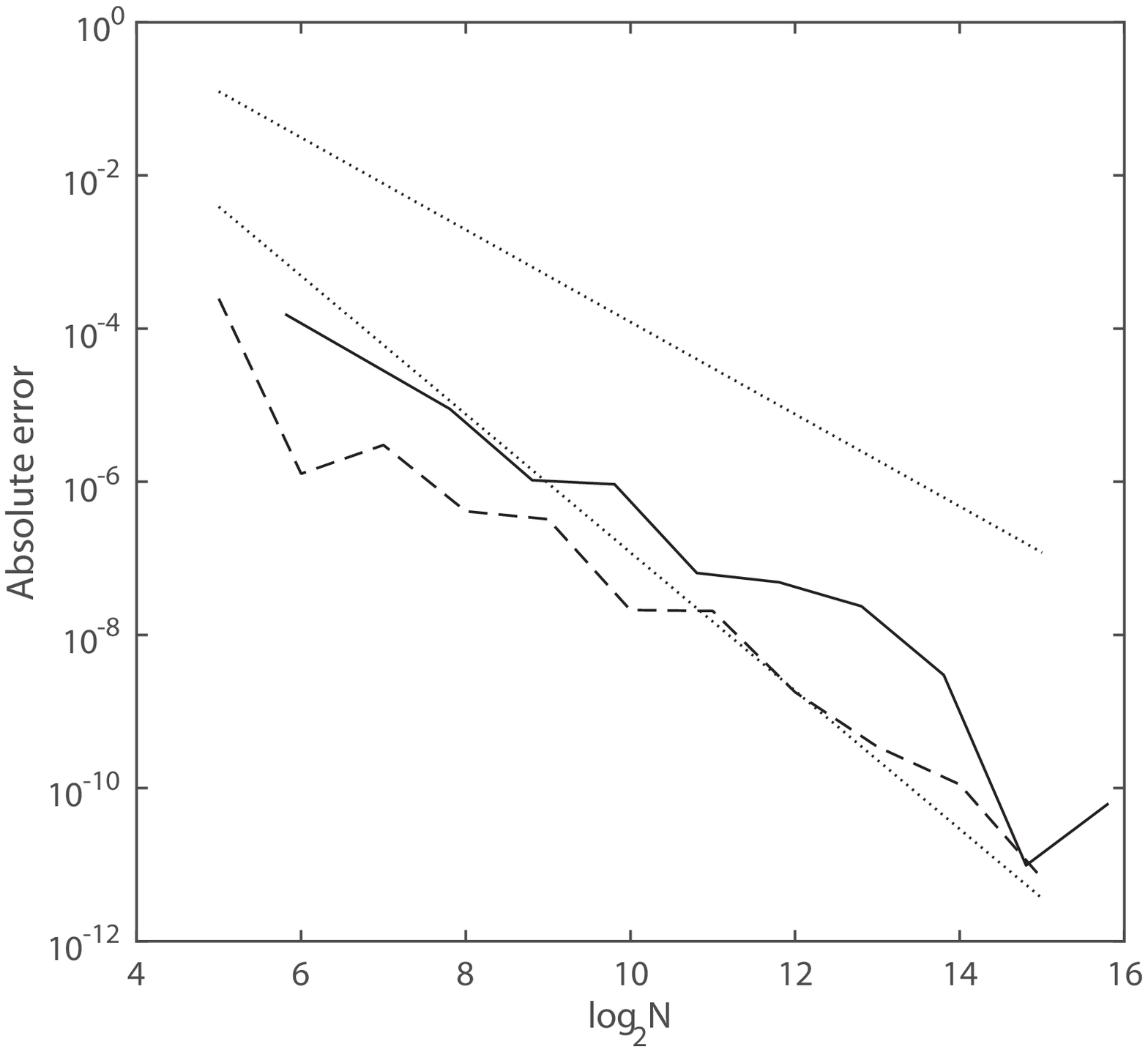}
\caption{The results for $f_1$ with $c_1=1.3$ (top), $f_2$ with $c_2=1$ (middle), and $f_2$ with $c_2=2$ (bottom) by using extrapolated polynomial lattice rules (solid) and interlaced polynomial lattice rules (dashed) with $\alpha=2$ (left) and $\alpha=3$ (right).}
\label{fig:test1}
\end{figure}

Figure~\ref{fig:test1} shows the results for the cases $\alpha=2$ (left column) and $\alpha=3$ (right column). Each row corresponds to the results for $f_1$ with $c_1=1.3$, $f_2$ with $c_2=1$, and $f_2$ with $c_2=2$, respectively. Again, for reference, the dotted lines correspond to $O(N^{-1})$ and $O(N^{-2})$ convergences for $\alpha=2$, and to $O(N^{-2})$ and $O(N^{-3})$ convergences for $\alpha=3$. For the case $\alpha=2$, extrapolated polynomial lattice rules perform competitively with interlaced polynomial lattice rules and achieve approximately the desired rate of the error convergence $O(N^{-2})$. For the case $\alpha=3$, similarly to the result for the bi-variate test function, interlaced polynomial lattice rules outperform extrapolated polynomial lattice rules, but the rate of the error convergence for extrapolated polynomial lattice rules improves as the number of points increases.

These numerical results indicate that extrapolated polynomial lattices rule can be quite useful in fast QMC matrix-vector multiplication with higher order convergence, which shall be undertaken in the near future. 

\section*{Acknowledgments}
The second author would like to thank Professor Josef Dick for his hospitality while visiting the University of New South Wales where most of this research was carried out.



\begin{thebibliography}{99}
\bibitem{BD09} {\sc J. Baldeaux and J. Dick}, \emph{QMC rules of arbitrary high order: reproducing kernel Hilbert space approach}, Constr. Approx. 30 (2009), pp.~495--527.
\bibitem{BDGP11} {\sc J. Baldeaux, J. Dick, J. Greslehner and F. Pillichshammer}, \emph{Construction algorithms for higher order polynomial lattice rules}, J. Complexity 27 (2011), pp.~281--299.
\bibitem{BDLNP12} {\sc J. Baldeaux, J. Dick, G. Leobacher, D. Nuyens and F. Pillichshammer}, \emph{Efficient calculation of the worst-case error and (fast) component-by-component construction of higher order polynomial lattice rules}, Numer. Algorithms 59 (2012), pp.~403--431.
\bibitem{DRbook} {\sc P.~J. Davis and P. Rabinowitz}, \emph{Methods of Numerical Integration}, Dover Publications, New York, 1984.
\bibitem{D08} {\sc J. Dick}, \emph{Walsh spaces containing smooth functions and quasi-Monte Carlo rules of arbitrary high order}, SIAM J. Numer. Anal. 46 (2008), pp.~1519--1553.
\bibitem{D09} {\sc J. Dick}, \emph{The decay of the Walsh coefficients of smooth functions}, Bull. Austral. Math. Soc. 80 (2009), pp.~430--453.
\bibitem{DKLNS14} {\sc J. Dick, F.~Y. Kuo, Q.~T. Le Gia, D. Nuyens and Ch. Schwab}, \emph{Higher order QMC Petrov-Galerkin discretization for affine parametric operator equations with random field inputs}, SIAM J. Numer. Anal. 52 (2014), pp.~2676--2702.
\bibitem{DKLS15} {\sc J. Dick, F.~Y. Kuo, Q.~T. Le Gia and Ch. Schwab}, \emph{Fast QMC matrix-vector multiplication}, SIAM J. Sci. Comput. 37 (2015), pp.~A1436--A1450.
\bibitem{DKS13} {\sc J. Dick, F.~Y. Kuo and I.~H. Sloan}, \emph{High-dimensional integration: The quasi-Monte Carlo way}, Acta Numer. 22 (2013), pp.~133--288.
\bibitem{DLS16} {\sc J. Dick, Q.~T. Le Gia and Ch. Schwab}, \emph{Higher order Quasi-Monte Carlo integration for holomorphic, parametric operator equations}, SIAM/ASA J. Uncertainty Quantification 4 (2016), pp.~48--79.
\bibitem{DP07} {\sc J. Dick and F. Pillichshammer}, \emph{Strong tractability of multivariate integration of arbitrary high order using digitally shifted polynomial lattice rules}, J. Complexity 23 (2007), pp.~436--453.
\bibitem{DPbook} {\sc J. Dick and F. Pillichshammer}, \emph{Digital Nets and Sequences: Discrepancy Theory and Quasi-Monte Carlo Integration}, Cambridge University Press, Cambridge, 2010.
\bibitem{GS16} {\sc R.~N. Gantner and Ch. Schwab}, \emph{Computational higher order quasi-Monte Carlo integration}, in: Monte Carlo and Quasi-Monte Carlo Methods, Springer Proceedings in Mathematics \& Statistics, vol 163, Springer, Heidelberg, 2016, pp.~271--288.
\bibitem{Gbook} {\sc W. Gautschi}, \emph{Numerical Analysis}, Birkh\"{a}user, Boston, 2012.
\bibitem{G15} {\sc T. Goda}, \emph{Good interlaced polynomial lattice rules for numerical integration in weighted Walsh spaces}, J. Comput. Appl. Math. 285 (2015), pp.~279--294.
\bibitem{G16} {\sc T. Goda}, \emph{Quasi-Monte Carlo integration using digital nets with antithetics}, J. Comput. Appl. Math. 304 (2016), pp.~26--42.
\bibitem{GD15} {\sc T. Goda and J. Dick}, \emph{Construction of interlaced scrambled polynomial lattice rules of arbitrary high order}, Found. Comput. Math. 15 (2015), pp.~1245--1278.
\bibitem{GSY16} {\sc T. Goda, K. Suzuki and T. Yoshiki}, \emph{Digital nets with infinite digit expansions and construction of folded digital nets for quasi-Monte Carlo integration}, J. Complexity 33 (2016), pp.~30--54.
\bibitem{LP17} {\sc V. Lemaire and G. Pag\`{e}s}, \emph{Multilevel Richardson--Romberg extrapolation}, Bernoulli 23 (2017), pp.~2643--2692.
\bibitem{N88} {\sc H. Niederreiter}, \emph{Low-discrepancy point sets obtained by digital constructions over finite fields}, Czechoslovak Math. J. 42 (1992), pp.~143--166.
\bibitem{Nbook} {\sc H. Niederreiter}, \emph{Random Number Generation and Quasi-Monte Carlo Methods}, CBMS-NSF Series in Applied Mathematics, vol. 63, SIAM, Philadelphia, 1992.
\bibitem{NC06a} {\sc D. Nuyens and R. Cools}, \emph{Fast algorithms for component-by-component construction of rank-1 lattice rules in shift-invariant reproducing kernel Hilbert spaces}, Math. Comp. 75 (2006), pp.~903--920.
\bibitem{NC06b} {\sc D. Nuyens and R. Cools}, \emph{Fast component-by-component construction, a reprise for different kernels}, in: Monte Carlo and quasi-Monte Carlo methods 2004, Springer, Berlin, 2006, pp.~373--387.
\bibitem{SJbook} {\sc I~H. Sloan and S. Joe}, \emph{Lattice Methods for Multiple Integration}, Oxford University Press, Oxford, 1994.
\bibitem{SW98} {\sc I.~H. Sloan and H. W\'ozniakowski}, \emph{When are quasi-Monte Carlo algorithms efficient for high-dimensional integrals?}, J. Complexity 14 (1998), pp.~1--33.
\bibitem{Y15} {\sc T. Yoshiki}, \emph{Bounds on Walsh coefficients by dyadic difference and a new Koksma-Hlawka type inequality for Quasi-Monte Carlo integration}, Hiroshima Math. J. 47 (2017), pp.~155--179.
\end{thebibliography}
\end{document}